\documentclass[authoryear,preprint,3p,review,11pt]{elsarticle}

\usepackage{amssymb}
\usepackage{amsthm}
\usepackage{amsmath}
\usepackage{amsfonts}
\usepackage{multirow}
\usepackage{bbm}
\usepackage{amsfonts}
\usepackage{footnote}
\usepackage{subcaption}

\usepackage{todonotes}
\usepackage{comment}
\newtheorem{theorem}{Theorem}
\newtheorem{lemma}[theorem]{Lemma}
\newtheorem{proposition}[theorem]{Proposition}
\newtheorem{corollary}[theorem]{Corollary}

\theoremstyle{definition}
\newtheorem{remark}[theorem]{Remark}
\newtheorem{definition}[theorem]{Definition}
\newtheorem{example}[theorem]{Example}

\begin{document}
	
	\begin{frontmatter}
		
		

		\author[label1]{Mohammed Essalih Benjrada \corref{c1}} %\fnref{thanks1}}
	\fntext[label1]{Department of Economics, University of Bergamo, Italy}
	\ead{mohammedessalih.benjrada@unibg.it}
	
	%\author[label1]{Tommaso Lando} %\fnref{thanks}}
%\fntext[label1]{Department of Economics, University of Bergamo, Italy}
%\ead{tommaso.lando@unibg.it}

%\author[label2]{Paulo Eduardo Oliveira} %\fnref{thanks1}} %etc.
%\fntext[label2]{CMUC, Department of Mathematics, University of Coimbra, Portugal}
%\ead{paulo@mat.uc.pt}


\title{Hazard Rate for Associated Data in Deconvolution Problems: Asymptotic Normality}%\tnoteref{thanktitle}\tnoteref{thanktitle1}}


%\tnotetext[thanks]{Tommaso Lando was supported by the Italian funds ex MURST 60\% 2019, by the Czech Science Foundation (GACR) under project 20-16764S and moreover by SP2020/11, an SGS research project of V\v{S}B-TU Ostrava. The support is greatly acknowledged.}
%\cortext[thanks1]{Tommaso Lando was supported by the Italian funds ex MURST 60\% 2019, by the Czech Science Foundation (GACR) under project 20-16764S and moreover by SP2020/11, an SGS research project of V\v{S}B-TU Ostrava. The support is greatly acknowledged.}

%\tnotetext[thanks1]{Idir Arab and Paulo Eduardo Oliveira were partially supported by the Centre for Mathematics of the University of Coimbra - UIDB/00324/2020, funded by the Portuguese Government through FCT/MCTES.}

%\author[1]{Idir Arab}
%\author[2]{Milto Hadjikyriakou}
%\author[1]{Paulo Eduardo Oliveira\corref{c1}}

\cortext[c1]{Corresponding author}

%\address[1]{CMUC, Department of Mathematics, University of Coimbra, Portugal}
%\address[2]{University of Central Lancashire, Cyprus}

\begin{abstract}
In reliability theory and survival analysis, observed data are often weakly dependent and subject to additive measurement errors. Such contamination arises when the underlying data are neither independent nor strongly mixed but instead exhibit association. This paper focuses on estimating the hazard rate by deconvolving the density function and constructing an estimator of the distribution function. We assume that the data originate from a strictly stationary sequence satisfying association conditions. Under appropriate smoothness assumptions on the error distribution, we establish the quadratic-mean convergence and asymptotic normality of the proposed estimators. The finite-sample performance of both the hazard rate and distribution function estimators is evaluated through a simulation study. We conclude with a discussion of open problems and potential future research directions.

\end{abstract}



\end{frontmatter}


%% main text\section{Introduction} % Initial capital letter, then lower case. No full stop.

\section{Introduction}
In reliability theory and survival analysis, estimating the hazard rate is a fundamental objective. The hazard rate is closely associated with \textit{stochastic aging conditions}, which characterize the general behavior of lifetime distributions. Specifically, the concepts of \textit{increasing hazard rate} (IHR) and \textit{decreasing hazard rate} (DHR) correspond to \textit{positive aging} and \textit{negative aging}, respectively. These notions describe whether aging has an adverse or beneficial effect on lifespan.  

\begin{itemize}
	\item In the IHR case, aging negatively impacts longevity, meaning that the risk of failure increases over time.  
	\item Conversely, in the DHR case, aging has a positive effect, with the risk of failure decreasing over time.
\end{itemize}

Accurately estimating the hazard rate and understanding its shape is of paramount importance in reliability theory and survival analysis. However, in practice, the ideal scenario of directly observed independent and identically distributed (i.i.d.) data is rarely available. Instead, key quantities are often estimated from data that exhibit weak dependence and are subject to measurement errors. A widely used approach in this context is the \textit{convolution model}, expressed as:

\begin{equation}
	Y_i = X_i + e_i.
	\label{convolution}
\end{equation}

This model, known as the \textit{convolution model}, describes the relationship between the observed variables \( \{ Y_i \}_{i=1}^{n} \), the latent variables \( \{ X_i \}_{i=1}^{n} \), and the measurement errors \( \{ e_i \}_{i=1}^{n} \), all of which are assumed to be continuous random variables. It accounts for the presence of noise in observed data and plays a crucial role in practical estimation procedures.

Let \( G(x) \) and \( g(x) \) denote the distribution and density functions of the observed process \( \{ Y_{i} \}_{i=1}^{n} \), while \( F(x) \) and \( f(x) \) correspond to the distribution and density functions of the latent process \( \{ X_{i} \}_{i=1}^{n} \). The measurement errors \( \{ e_{i} \}_{i=1}^{n} \) are assumed to be independent and identically distributed (i.i.d.) with a fully known density function \( r(x) \). Moreover, the errors are independent of both each other and the latent variables \( \{ X_{i} \}_{i=1}^{n} \). This assumption ensures the \textit{identifiability} of the model, meaning that the contamination mechanism does not depend on the underlying events.  

Furthermore, we assume that the latent process \( \{ X_{i} \}_{i=1}^{n} \) forms a strictly stationary sequence of positively associated random variables.

\begin{definition}
	A finite family of random variables $\left\{ X_{i}\right\}_{i=1}^{n}$ is said to be \textit{positively associated} (or simply \textit{associated}) if for any two real-valued, \textit{coordinate-wise increasing} functions $\Phi_{1}(\cdot)$ and $\Phi_{2}(\cdot)$ defined on $\mathbb{R}^{n}$, the following holds:  
	
	Whenever $\mathbb{E}\left[ \Phi_{j}^{2}(\left\{ X_{i}\right\}_{i=1}^{n}) \right] < +\infty$ for $j = 1, 2$, the covariance satisfies:
	\begin{equation*}
		\text{Cov}\left( \Phi_{1}(\left\{ X_{i}\right\}_{i=1}^{n}), \Phi_{2}(\left\{ X_{i}\right\}_{i=1}^{n}) \right) \geq 0.
	\end{equation*}
\end{definition}

An infinite process $\left\{ X_{i}\right\}_{i=1}^{+\infty}$ is said to be \textit{associated} if every finite subcollection of its random variables is associated. This concept of associated random variables was introduced by \cite{Esary1967}, primarily for applications in reliability theory. The following properties of associated random variables were also established by \cite{Esary1967}:

\begin{enumerate}
	\item \textbf{Property $P_{1}$}: Any finite subcollection of associated random variables is itself associated.
	
	\item \textbf{Property $P_{2}$}: The union of two or more independent sets of associated random variables remains associated.
	
	\item \textbf{Property $P_{3}$}: A singleton set consisting of a single random variable is associated.
	
	\item \textbf{Property $P_{4}$}: Non-decreasing functions preserve association. Specifically, if $U = \left\{ U_{i}\right\}_{i=1}^{n}$ is an associated random process, then $\left\{ h_{j}\left( U\right) \right\}_{j=1}^{m}$ is also associated, where each $h_{j}(\cdot)$ is a non-decreasing function.
\end{enumerate}

The concept of association is of significant importance in various fields. For instance, in reliability theory, random variables often represent the lifetimes of components that are associated but not necessarily independent. 

To illustrate this, suppose $\left\{ X_{i}\right\}_{i=1}^{n}$ is a collection of associated random variables. Then, the observed process $\left\{ Y_{i}\right\}_{i=1}^{n}$, where $Y_{i} = X_{i} + e_{i}$ and $e_{i}$ represents an error term, is also associated. This can be shown as follows:

\begin{enumerate}
	\item By \textbf{Property $P_{3}$}, each singleton $e_{i}$ for $i = 1, \dots, n$ is associated.
	\item Since the $e_{i}$'s are independent random variables, \textbf{Property $P_{2}$} implies that the collection $\left\{ e_{i}\right\}_{i=1}^{n}$ is associated.
	\item Using \textbf{Property $P_{2}$} again, the union $\left\{ X_{i}\right\}_{i=1}^{n} \cup \left\{ e_{i}\right\}_{i=1}^{n}$ is associated.
	\item Finally, applying \textbf{Property $P_{4}$} with the non-decreasing functions $h_{j}\left( \left\{ X_{i}\right\}_{i=1}^{n} \cup \left\{ e_{i}\right\}_{i=1}^{n} \right) = X_{j} + e_{j}$ for $j = 1, \dots, n$, we conclude that $\left\{ Y_{i} = X_{i} + e_{i}\right\}_{i=1}^{n}$ is an associated random process.
\end{enumerate}

This result implies that if the lifetimes of associated components in a system are subject to measurement errors (due to experimental conditions or tools), the observed lifetimes remain associated. For further reading on the concept of association, we recommend the monographs by \cite{bulinski2007limit}, \cite{Oliveira2012}, and \cite{Rao2012}.

Such models with contaminated noise are prevalent in various domains, including biological organisms, communication theory, biostatistics, and other fields. For example, in the context of biological organisms, particularly in AIDS studies, the following scenario arises:

For the \(i^{\text{th}}\) individual:
\begin{itemize}
	\item \(Y_i\) represents the time from a starting point (e.g., initial observation) to the appearance of symptoms.
	\item \(e_i\) denotes the time from the starting point to the occurrence of infection.
	\item \(X_i\) corresponds to the time from infection to the appearance of symptoms (i.e., the incubation period).
\end{itemize}

In this case, the observed data \(\left\{ Y_{i} = X_{i} + e_{i} \right\}_{i=1}^{n}\) provides incomplete or corrupted information about the true incubation period \(X_i\), as it is influenced by the random infection time \(e_i\).

In reliability theory and related fields, the random variables \(\left\{ X_{i} \right\}_{i=1}^{n}\) often represent the lifetimes of \(n\) components or organisms. In such contexts, the hazard rate \(\lambda(x)\) plays a pivotal role, as it fully characterizes the distribution of the event under study (e.g., time to failure or death). The hazard rate function \(\lambda(x)\) is defined as:

\[
\lambda(x) := \lim_{dx \rightarrow 0} \frac{P\left( x \leq X < x + dx \mid X \geq x \right)}{dx} = \frac{f(x)}{1 - F(x)},
\]

where:
\begin{itemize}
	\item \(f(x)\) is the probability density function (PDF) of \(X\),
	\item \(F(x)\) is the cumulative distribution function (CDF) of \(X\).
\end{itemize}

The quantity \(\lambda(x) \, dx\) represents the probability that an organism or component, which is functioning at time \(x\), will fail within the infinitesimal interval \([x, x + dx]\) as \(dx \rightarrow 0\). 

The hazard rate \(\lambda(x)\) is particularly informative compared to other characterizing functions (such as the density or distribution functions). For instance, its graph can reveal key properties of the distribution, including:
\begin{itemize}
	\item The \textbf{mode(s)} (peaks of the hazard rate),
	\item \textbf{Symmetry} (behavior of the hazard rate over time),
	\item \textbf{Dispersion} (spread of the hazard rate),
	\item \textbf{Flattening} (how the hazard rate changes over time).
\end{itemize}

The estimation strategy considered in this paper is to estimate \(\lambda(x)\) using a deconvolving kernel-type density estimator, \( f_n(x) \), along with an estimator of the distribution function, \( F_n(x) \), which is obtained by integrating \( f_n(x) \). This approach naturally falls within the deconvolution framework. The deconvolution problem has been extensively studied in the literature, with most works focusing on estimating the unknown density and determining the rate of convergence under specific error structures.  

\cite{benjrada2022deconvolving} proposed an estimator for the distribution function by integrating the density estimator and examined its asymptotic normality, assuming that the measurement error's tail exhibits either a super-smooth or ordinary-smooth behavior. \cite{Fan1991a} investigated the kernel density estimator based on i.i.d. copies of the random variables \( Y_i \) and analyzed its asymptotic normality. \cite{Masry2003} extended this work to the multivariate setting, considering deconvolving density estimators constructed from associated observations, and established both their quadratic mean properties and asymptotic normality. Further foundational contributions can be found in \cite{Wise1977}, \cite{Liu1989}, \cite{Snyder1988}, among others.

Regarding the hazard rate function, \cite{Comte2018} studied the case where the random variables are subject to both censoring and measurement errors. Specifically, the measurement errors are assumed to affect both the variable of interest and the censoring variable. The authors described the model and estimation strategy in detail and derived an \( L_2 \)-risk bound for the proposed estimator.  

\cite{benjrada2022hazard} estimated the hazard rate function using a deconvolving kernel density estimator, \( f_n(x) \), and a distribution function estimator, \( F_n(x) \), obtained by integrating \( f_n(x) \). They established strong uniform consistency (with convergence rates) for \( f_n(x) \), \( F_n(x) \), and \( \lambda_n(x) \) separately.  

To the best of our knowledge, the asymptotic normality and quadratic mean convergence of the hazard rate function under the corrupted-associated model have not been established. This gap in the literature motivates the present study.  

The remainder of this paper is organized as follows:  
\begin{itemize}
	\item \textbf{Section 2} introduces the necessary notations, definitions, and assumptions.  
	\item \textbf{Section 3} presents the main theoretical results.  
	\item \textbf{Section 4} discusses future directions.  
	\item \textbf{Section 5} provides experimental studies to illustrate the estimator's performance.  
	\item \textbf{Section 6} contains the proofs of the main results.  
\end{itemize}

\section{Estimates and Assumptions}
\subsection{Estimation}
By a straightforward generalization, the hazard rate estimator  
\( \hat{\lambda}_{n}(x) \) was introduced by \cite{Watson1964} in the error-free case  
(i.e., when no measurement errors are present) as follows:  
\begin{equation}
	\hat{\lambda}_{n}(x) = \frac{\hat{f}_{n}(x)}{1 - \hat{F}_{n}(x)},
	\label{estimat}
\end{equation}
where \( \hat{f}_{n}(x) \) denotes the kernel-type density estimator, and  
\( \hat{F}_{n}(x) \) represents the empirical distribution function.

In developmental process fields, convolution models are not only of theoretical interest but also have practical implications. Ignoring measurement errors can introduce substantial bias, leading to incorrect conclusions. As a result, we do not focus on (\ref{estimat}), since \( \hat{f}_{n}(x) \) and \( \hat{F}_{n}(x) \) do not account for the presence of measurement errors.  

For an arbitrary function \( \zeta(\cdot) \), we define its corresponding Fourier transform as follows:  
\[
\phi _{\zeta }(t) := \int_{-\infty }^{+\infty } e^{- \text{i}ts} \zeta(s) \, ds.
\]  

The deconvolving kernel estimator of \( f(x) \) is given by (see \cite{Liu1989}):  
\begin{equation}  \label{fn}
	f_{n}(x) = \frac{1}{nh_{n}} \sum_{i=1}^{n} W_{h_{n}} \left( \frac{x - Y_{i}}{h_{n}} \right),
\end{equation}
where  
\[
W_{h_{n}}(x) = \frac{1}{2\pi} \int_{-\infty }^{+\infty } e^{-itx} \frac{\phi _{k}(t)}{\phi _{r}(t/h_{n})} \, dt.
\]  
Here, \( k(\cdot) \) is a smooth probability kernel, and \( \{ h_{n} \}_{n\geq 1} \) is a bandwidth sequence.

\begin{remark}
	\label{pdf}  
	It is clear that \( f_{n}(t) \) has the form of a classical kernel density estimator; however, here the kernel \( W_{h_{n}}(t) \) is related to the bandwidth sequence \( h_{n} \). Note that \( W_{h_{n}}(t) \) always lies in \( L_{1}(\mathbb{R}) \). In fact, by equation (\ref{fn}) and the inversion formula, we obtain
	\[
	\int_{-\infty }^{+\infty } e^{itx} W_{h_{n}}(x) \, dx = \frac{\phi _{k}(t)}{ \phi _{r}(t/h_{n})}.
	\]
	Taking \( t = 0 \) gives
	\[
	\int_{-\infty }^{+\infty } W_{h_{n}}(t) \, dt = 1.
	\]
	Hence, \( W_{h_{n}}(t) \) is a kernel, and we have \( \int_{-\infty }^{+\infty } f_{n}(x) \, dx = 1 \).
\end{remark}

A plug-in estimation of the cumulative distribution function (CDF) was first introduced by \cite{Fan1991}, who proposed a consistent estimator \( F_n \) within the framework of deconvolution methods, given by:  

\begin{equation}
	F_n(x) = \int_{-\infty}^{x} f_n(t) dt = \frac{1}{n} \sum_{i=1}^{n} M_{h_n} \left( \frac{x - Y_i}{h_n} \right) = \frac{1}{n} \sum_{i=1}^{n} M_{h_n, i} (x),  
	\label{Fn}
\end{equation}

where \( M_{h_n}(x) = \int_{-\infty}^{x} W_{h_n}(t) dt \).  

The smooth cumulative density estimator in (\ref{Fn}) is commonly known as the \textit{kernel distribution estimator}, originally proposed by \cite{Nadaraya1964} as an alternative to the empirical distribution function in the case of error-free data. Following a similar approach, \cite{Fan1991} extended this methodology to the deconvolution setting, providing a consistent estimator of \( F_n \). Throughout this paper, we adopt the same approach.

To adapt the classical hazard rate estimator in Eq. (\ref{estimat}) for contaminated data, we replace \( \hat{f}_n(\cdot) \) and \( \hat{F}_n(\cdot) \) with \( f_n(\cdot) \) and \( F_n(\cdot) \), respectively. The resulting estimator is given by:  

\begin{equation}\label{est1}
	\lambda_n(x) = \frac{f_n(x)}{1 - \min(F_n(x), 1 - \epsilon)}.
\end{equation}

where \( 0 < \epsilon < 1 \) and \( \epsilon \) approaches 0 from above. The inclusion of \( \epsilon \) in the denominator serves as a regularization term to prevent numerical instabilities when dividing by near-zero values. For large \( n \), the estimator can be approximated as  

\begin{equation}\label{est2}  
	\lambda_n(x) = \frac{f_n(x)}{1 - F_n(x)}.  
\end{equation}  

From \cite{benjrada2022hazard}, it follows that \( F_n(x) \to F(x) \) almost surely (a.s.), where \( F(x) < 1 \) almost everywhere (a.e.). Consequently, the estimator in (\ref{est1}) can be reduced to (\ref{est2}) when dealing with large sample sizes. This implies that the estimators in (\ref{est1}) and (\ref{est2}) can be used interchangeably, depending on whether \( n \) is small or sufficiently large (asymptotic analysis). However, in Section 3, we adopt the latter estimator (\ref{est2}) since our focus is on asymptotic normality.  

The approach in (\ref{est2}) was first introduced in \cite{Comte2018} to estimate the length of women's pregnancies based on incomplete and contaminated data.

\subsection{Assumptions and Notation}

Throughout this work, all constants denoted by $C$ represent generic positive constants that may vary from line to line. To streamline the presentation, we introduce the following assumptions:

\subsection*{(\textbf{H1}) Assumptions on the Error Density}

The error density function $r(\cdot)$ is assumed to be known. Additionally, its characteristic function $\phi_r(\cdot)$ satisfies:  

\begin{enumerate}
	\item $|\phi_r(t)| > 0$ for all $t \in \mathbb{R}$.
	\item $\lim\limits_{t \to +\infty} t^{\beta} \phi_r(t) = \beta_1$, where $\beta$ is an even number and $\beta_1$ is a positive constant.
\end{enumerate}

\subsection*{(\textbf{H2}) Properties of the Kernel Function}

The kernel function $k(\cdot)$ is a bounded density with an even Fourier transform $\phi_k(t)$ and satisfies:

\begin{enumerate}
	\item $\int_{-\infty}^{+\infty} s k(s) \, ds = 0$.
	\item $\int_{-\infty}^{+\infty} s^2 k(s) \, ds < +\infty$.
	\item $\phi_k(t) = O(t)$ as $t \to 0$.
\end{enumerate}

\subsection*{(\textbf{H3}) Integrability Conditions on the Kernel's Fourier Transform}

The Fourier transform $\phi_k(t)$ satisfies the following integrability conditions:

\begin{enumerate}
	\item $\int_{-\infty}^{+\infty} |t^{\beta +1} \phi_k(t)| \, dt < +\infty$.
	\item $\int_{-\infty}^{+\infty} |t^{\beta} \phi_k(t)| \, dt < +\infty$.
	\item $\int_{-\infty}^{+\infty} |t^{\beta -1} \phi_k(t)| \, dt < +\infty$.
	\item $\int_{-\infty}^{+\infty} |t^{2\beta} \phi_k(t)^2| \, dt < +\infty$.
\end{enumerate}

\subsection*{(\textbf{H4}) Regularity Conditions on $\phi_k(t)$}

The function $\phi_k(t)$ and its derivatives satisfy:

\begin{enumerate}
	\item $\int_{-\infty}^{+\infty} |t|^{\beta-2} |\phi_k(t)| \, dt < +\infty$.
	\item $\int_{-\infty}^{+\infty} |t|^{\beta-1} |\phi_k'(t)| \, dt < +\infty$.
	\item $\int_{-\infty}^{+\infty} |t|^{\beta} |\phi_k''(t)| \, dt < +\infty$.
\end{enumerate}

\subsection*{\textbf{(H5)} Conditions on the Marginal Distribution and Dependence}

\begin{enumerate}
	\item The marginal density function $g(x)$ is bounded on $\mathbb{R}$.
	\item The sequence $\{X_j\}_{j \geq 1}$ is an associated process, and the covariance structure satisfies:
	\[
	\sum\limits_{j\geq 2} j^{\mu} \text{Cov}(X_1, X_j) < +\infty, \quad \text{for some } \mu > 1.
	\]
\end{enumerate}

\subsection*{(\textbf{H6}) Growth Conditions on the Sequences $\{q_n\}$ and $\{u_n\}$}

Let $\{q_n\}$ and $\{u_n\}$ be sequences of positive integers tending to infinity, satisfying:
\[
q_n u_n = o\left( (n h_n)^{1/2} \right),
\]
and
\[
\frac{1}{n h_n^{2(\beta +1)}} \sum\limits_{i \geq p_n} \text{Cov}(X_1, X_i) \to 0,
\]
where  
\[
p_n = \left\lfloor \frac{(n h_n)^{1/2}}{u_n} \right\rfloor.
\]

	We introduce the following notations:  
	\begin{equation}
		L(u) := \frac{1}{\pi \beta_1} \int_{0}^{+\infty} \cos (t u) t^{\beta} \phi_k(t) dt.  
		\label{L}
	\end{equation}
	
	where \( \beta_1 \) and \( \beta \) are parameters defined in assumption (\textbf{H1}).  
	
	\[
	\kappa_{h_n} (x) := \left[ h_n^{\beta} W_{h_n}(x) \right]^2.
	\]
	
	We also define the constants:  
	
	\begin{equation}
		D_1 := \int_{-\infty}^{+\infty} L(u) \, du.  
		\label{D1}
	\end{equation}
	
	\begin{equation}
		D_2 := \frac{1}{2\pi \beta_1^2} \int_{-\infty}^{+\infty} \left\vert t^{2\beta} \phi_k(t)^2 \right\vert dt.  
		\label{D2}
	\end{equation}

\subsection{Comments on hypotheses}  

The positive constant \( \beta \) in (\textbf{H1}), (\textbf{H3}), (\textbf{H4}),  
and (\textbf{H6}) represents the order of the contaminating density \( r(x) \).  
This parameter plays a vital role in determining the rate of convergence of the hazard rate estimator. Assumption (\textbf{H1}) refers to the ordinary smooth case of order \( \beta \).  
For example, the Gamma\((a, b)\) and Laplacian distributions fall under the ordinary smooth case  
with \( \beta = a \) and \( \beta = 2 \), respectively. Assumption (\textbf{H2}) presents a classical choice of the kernel in non-parametric estimation. Assumption (\textbf{H3}) is used to establish the \( L^2 \)-norm of \( W_{h_n} \)  
and to derive bounds for the \( L^{\infty} \)-norm. Assumption (\textbf{H4}) is required to establish an approximation of the identity  
(see the asymptotic approximation in Lemma \ref{ident})  
with the aim of deriving asymptotic variances.  
Note that (\textbf{H3}) and (\textbf{H4}) can be relaxed if \( \phi_k \) is compactly supported. For more details, see \cite{bissantz2007}. Assumption (\textbf{H5}) is standard in non-parametric estimation from dependent data,  
as it facilitates the calculation of the limit covariance in the associated concepts. Finally, the conditions in assumption (\textbf{H6}) are frequently used when establishing  
asymptotic normality for dependent random variables.  
In particular, they allow the application of the well-known big block and small block techniques.  

The assumption  $\sum\limits_{j \geq 2} j^{\mu} \text{Cov}(X_1, X_j) < +\infty$ in (\textbf{H5}) holds if the covariance sequence satisfies  $\text{Cov}(X_1, X_j) = O(j^{-k})$ for some \( k > \mu + 1 \). Since we assume \( \mu > 1 \), this condition implies that \( k > 2 \), ensuring the summability of the series. We now provide two examples to illustrate this condition.
\begin{example} \label{Example1}
	As a first example, consider the AR(1) process:  
	\begin{equation*}
		X_{j} = \phi_1 X_{j-1} + \varepsilon_{j},
	\end{equation*}  
	where \( \{ \varepsilon_i \}_{i\geq 1} \) is an i.i.d. sequence of random variables with zero mean and unit variance.  
	
	In this case, the process \( \{X_{j}\} \) is positively associated, and its covariance function is given by:  
	\begin{equation*}
		\text{Cov}(X_1, X_j) = \phi_1^{j-1}.
	\end{equation*}  
	Thus, Condition (\textbf{H5})-2 is satisfied whenever \( 0<\phi_1 < 1 \). Indeed, in this case, the covariance decays exponentially as \( j \to +\infty \):  
	\begin{equation*}
		\text{Cov}(X_1, X_j) = O(e^{-c(j-1)}),
	\end{equation*}  
	where \( c = -\log \phi_1 \). Hence, since the series  
	\begin{equation*}  
		\sum\limits_{j \geq 2} j^{\mu} e^{-c(j-1)}  
	\end{equation*}  
	is power-exponentially weighted, it converges for all real \( \mu \) whenever the exponential decay \( e^{-c(j-1)} \) dominates the algebraic growth \( j^{\mu} \). Given our assumption that \( \mu > 1 \), we can choose \( c \) in accordance with \( \phi_{1} \) to ensure this dominance, thereby guaranteeing absolute summability and fulfilling the required condition.
\end{example}

\begin{example} \label{Example2}  
	Next, consider the Moving Average process of infinite order, denoted as \( MA(\infty) \). For a sequence of positive coefficients \( (\alpha_i)_{i \geq 1} \), the process \( \{X_j\} \) is defined as: $X_j = \sum\limits_{i=0}^{+\infty} \alpha_i \varepsilon_{j-i}.$ This process is positively associated, as established by properties P4--P6 in Esary et al. \cite{Esary1967}. The covariance sequence \( \text{Cov}(X_1, X_j) \) is given by:  
	\[
	\text{Cov}(X_1, X_j) = \sum\limits_{i=0}^{+\infty} \alpha_i \alpha_{j+i-1}.
	\]  
	By specifying the coefficients as \( \alpha_i = O(i^{-\delta}) \), we obtain: $\text{Cov}(X_1, X_j) = O(j^{-2\delta +1})$. For the summability condition to hold, we require: $-2\delta +1 < -\mu -1,$ which simplifies to: $\delta > \frac{\mu+2}{2}.$ Since \( \mu > 1 \), this condition implies \( \delta > 3/2 \), ensuring that \( \alpha_i \) decays faster than \( i^{-3/2} \).  
\end{example}

As a first result of the deconvolutional approach, we give the following proposition:

\begin{proposition}
	\label{P1} 
	
	1) For all $x\in 
	%TCIMACRO{\U{211d} }%
	%BeginExpansion
	\mathbb{R}
	%EndExpansion
	$, we have 
	
	\begin{equation*}
		\lim_{n\rightarrow +\infty }E[f_{n}(x)]=f(x),
	\end{equation*}

	\begin{equation*}
		\lim_{n\rightarrow +\infty }E[F_{n}(x)]=F(x).
	\end{equation*}
	
	2) Under conditions (\textbf{H2})-1, (\textbf{H2})-2, and supposing that $f
	$ and $F$ are in $C_{2}\left( 
	%TCIMACRO{\U{211d} }%
	%BeginExpansion
	\mathbb{R}
	%EndExpansion
	\right) $, we find
	\begin{equation*}
		\lim_{n\rightarrow +\infty }\left( h_{n}\right) ^{-2}E\left[ f_{n}(x)-f(x)
		\right] =\frac{1}{2}f^{\prime \prime }(x)\int_{-\infty }^{+\infty
		}s^{2}k(s)ds,
	\end{equation*}
	\begin{equation*}
		\lim_{n\rightarrow +\infty }\left( h_{n}\right) ^{-2}E\left[ F_{n}(x)-F(x)
		\right] =\frac{1}{2}F^{\prime \prime }(x)\int_{-\infty }^{+\infty
		}s^{2}k(s)ds.
	\end{equation*}

\end{proposition}

The proof of Proposition \ref{P1} is rather technical (see Proposition 1 in \cite{benjrada2022hazard} and Proposition 1 in \cite{Masry2003}) and then omitted.

\section{Main Results}

This section presents the main results of the paper. For clarity, we divide it into two subsections. The first subsection discusses the quadratic mean convergence of \( \lambda_n(x) \). The second subsection establishes its convergence in distribution toward a Gaussian distribution, leveraging the \(\delta\)-method theorem from \cite{Doob1935}. As a direct consequence, an asymptotic confidence interval is constructed in Corollary \ref{CI}.

\subsection{The Quadratic Mean Convergence}

First, we derive the precise asymptotic expression for \(\text{Var}(\lambda_n(x))\). Since \(\lambda_n(x)\) is defined as the ratio of \( f_n(x) \) to \( 1 - F_n(x) \), we need to compute \(\text{Cov}(f_n(x), 1 - F_n(x))\) and \(\text{Var}(1 - F_n(x))\) separately. The expression for \(\text{Var}(f_n(x))\) has already been established by \cite{Masry2003} for this type of data.

\begin{theorem}
	\label{variance} 
	Under the hypotheses (\textbf{H1}), (\textbf{H2})-3, (\textbf{H3})--(\textbf{H5}), suppose also that $\left| g_{Y_{1},Y_{i}}(u,v) - g(u) g(v) \right| \leq M < +\infty$	and $nh_{n}^{2\beta} \to +\infty \quad \text{as} \quad n \to +\infty$. Then, we have:  
	\begin{equation}
		\lim_{n\rightarrow +\infty } nh_{n}^{2\beta } \, \text{Var}(1 - F_{n}(x)) = D_{1}^{2} G(x) \left[ 1 - G(x) \right],  \label{var1}
	\end{equation}
	\begin{equation}
		\lim_{n\rightarrow +\infty } nh_{n}^{2\beta } \, \text{Cov}(f_{n}(x), 1 - F_{n}(x)) = -D_{2} G(x),  \label{cov1}
	\end{equation}
	where \( G(x) \) is the distribution function of \( \left\{ Y_{i} \right\}_{i=1}^{n} \).
\end{theorem}

Now, we are in a position to evaluate the asymptotic expression of $\text{Var}\left( \lambda _{n}\left( x\right) \right) $.

\begin{corollary}
	\label{corellary1} Under the assumptions of Theorem \ref{variance}, we have 
	\begin{eqnarray*}
		\lim_{n\rightarrow +\infty }nh_{n}^{2\beta }\text{Var}(\lambda _{n}(x)) &=&\frac{%
			D_{2}g\left( x\right) }{\left[ 1-F(x)\right] ^{2}}-2\frac{D_{2}f\left(
			x\right) G\left( x\right) }{\left[ 1-F(x)\right] ^{3}}+\frac{D_{1}^{2}\left[
			f(x)\right] ^{2}G(x)\left[ 1-G(x)\right] }{\left[ 1-F(x)\right] ^{4}} \\
		&&\left. =:\sigma ^{2}\left( x\right) \right. .
	\end{eqnarray*}
\end{corollary}

\begin{remark}
	The asymptotic variance $\sigma ^{2}\left( x\right) $ depends on $ D_{1}$ and $D_{2}$, which pertain to $\phi _{k}\left(.\right) $. Thus, in order to have a hazard rate estimator with minimal variance, one minimizes $ \sigma ^{2}\left(x\right) $ by dealing with $D_{1}$ and $D_{2}$ along with $\phi _{k}\left(.\right) $ in a certain class of kernels.
\end{remark}

As to the asymptotic expectation, the next centering term is needed: 
\begin{equation*}
	\tilde{E}(\lambda _{n}(x))=\frac{E\left[ f_{n}(x)\right] }{1-E\left[ F_{n}(x)%
		\right] }.  
\end{equation*}%
We give the following intermediate result.

\begin{proposition} \label{bais1} Under the assumptions of Theorem \ref{variance}, we have
	\begin{equation*}
		E(\lambda _{n}(x))=\tilde{E}(\lambda _{n}(x))+O((nh_n^{2\beta})^{-1}).
	\end{equation*}
	
\end{proposition}

Next
\begin{proposition}
	\label{bais2} Under the assumptions of Theorem \ref{variance} and the
	assumption that $f$ and $F$ are in $C_{2}(%
	%TCIMACRO{\U{211d} }%
	%BeginExpansion
	\mathbb{R}
	%EndExpansion
	)$, we get 
	\[
	\tilde{E}(\lambda _{n}(x))-\lambda (x)=\frac{h_{n}^{2}}{2}\left[ \lambda ^{\prime
		\prime }(x)-\frac{2\lambda ^{\prime }(x)F^{\prime }(x)}{1-F(x)}\right]
	\int_{-\infty }^{+\infty }t^{2}k(t)dt+o(1).
	\]
\end{proposition}

\subsection{The Asymptotic Normality}
The fundamental result in this subsection is Theorem \ref{an} hereafter.

\begin{theorem}
	\label{an} 
	Under assumptions \textbf{(H1)}--\textbf{(H6)}, suppose that: $nh_n^{2\beta} \to +\infty \quad \text{as} \quad n \to +\infty.$ 
	Then, the following asymptotic normality result holds:  
	\begin{equation*}
		\sqrt{nh_n^{2\beta}} \left( \lambda_n(x) - \mathbb{E} \left[ \lambda_n(x) \right] \right) 
		\overset{\mathcal{D}}{\longrightarrow} \mathcal{N} \left( 0, \sigma^2(x) \right).
	\end{equation*}
\end{theorem}

By applying Propositions \ref{bais1} and \ref{bais2}, and assuming $nh_n^{2\beta} \to +\infty \quad \text{as} \quad n \to +\infty,$ we obtain:  $E(\lambda_{n}(x)) = \lambda(x) + O(h_n^2) + O\left((nh_n^{2\beta})^{-1}\right) =\lambda(x)+ o(1).$ Combining this result with Theorem \ref{an}, we obtain the following corollary.

\begin{corollary}
	\label{better} Under the assumptions of Theorem \ref{an} and the assumption that $%
	f$ and $F$ are in $C_{2}(%
	%TCIMACRO{\U{211d} }%
	%BeginExpansion
	\mathbb{R}
	%EndExpansion
	)$ we find 
	\begin{equation*}
		\sqrt{nh_{n}^{2\beta }}\left( \lambda _{n}\left( x\right) -\lambda \left(
		x\right) \right) \overset{\mathcal{D}}{\rightarrow }\mathcal{N}\left(
		0,\sigma^2 \left( x\right) \right) .
	\end{equation*}
\end{corollary}

\begin{remark}
	\label{remark3} The rate in Corollary \ref{better} is still slower compared to that asserted for the free-error framework in the i.i.d. and mixing
	concepts. Note that the resulting rate depend on the parameter $\beta $
	which pertains to the contamination mechanism. In addition, we remark that
	the larger $\beta $ is, the slower the rate of convergence become.
\end{remark}

Using the results in Corollary \ref{better}, Corollary \ref{CI} bellow gives
the confidence interval of asymptotic level $1-\zeta $ for $\lambda \left( x\right) $. To do so, we estimate the asymptotic variance $\sigma
^{2}\left( x\right) $\ using the following plug-in estimate: 
\[
\sigma _{n}^{2}\left( x\right) =\frac{D_{2}g_{n}\left( x\right) }{\left[
	1-F_{n}\left( x\right) \right] ^{2}}-2\frac{D_{2}f_{n}\left( x\right)
	G_{n}\left( x\right) }{\left[ 1-F_{n}\left( x\right) \right] ^{3}}+\frac{%
	D_{1}^{2}\left[ f_{n}\left( x\right) \right] ^{2}G_{n}\left( x\right) \left[
	1-G_{n}\left( x\right) \right] }{\left[ 1-F_{n}\left( x\right) \right] ^{4}},
\]

where $g_{n}\left( x\right) $ is a classical kernel type estimation of $%
g\left( x\right) $ and $G_{n}\left( x\right) :=\int_{-\infty
}^{x}g_{n}\left( u\right) du$.

\begin{corollary}
	\label{CI} Under the assumptions of Corollary \ref{better}, we can construct
	confidence intervals $IC_{\lambda _{n}}\left( x\right) $ for $\lambda \left( x\right) $ given by 
	\[
	IC_{\lambda _{n}}\left( x\right) =\left] \lambda _{n}\left( x\right)
	-z_{\left( 1-\zeta /2\right) }\frac{\sigma _{n}\left( x\right) }{\sqrt{%
			nh_{n}^{2\beta }}},\lambda _{n}\left( x\right) +z_{\left( 1-\zeta /2\right) }%
	\frac{\sigma _{n}\left( x\right) }{\sqrt{nh_{n}^{2\beta }}}\right[ ,
	\]
	
	where $z_{\left( 1-\zeta /2\right) }$ stands for the $\left( 1-\zeta
	/2\right) -$quantile of the standard normal distribution.
\end{corollary}

\section{Future Directions}

This section explores several open questions and potential avenues for future research. We focus on two primary areas: 1) the case of supersmooth error, and 2) the case of non-standard measurement error.

\subsection{The Case of Supersmooth Error}

One promising direction for future research is the estimation of the hazard rate function from data corrupted by supersmooth error. Supersmooth error refers to scenarios where the characteristic function \(\phi_{r}\) of the measurement error decays exponentially fast as \(|t| \rightarrow +\infty\). Specifically, this occurs when the following condition is satisfied:
\[
\beta_{2}e^{-m|t|^{\alpha}}|t|^{\beta} \leq |\phi_{r}(t)| \leq \beta_{3}e^{-m|t|^{\alpha}}|t|^{\beta},
\]
where \(\alpha\), \(m\), \(\beta_2\), and \(\beta_3\) are positive constants, and \(\beta\) is a real number.

Supersmooth distributions include, for example, the Cauchy, Mixture Normal, and Normal densities. The constant \(\beta\) represents the order of the noise density \(r(x)\) and directly influences the convergence rate of the estimate \(\lambda_{n}(x)\). The smoother the error distribution, the slower the convergence rate of the estimator.

Deconvolution challenges are closely tied to the smoothness of the error distribution. Supersmooth distributions are notably more difficult to deconvolve than ordinary smooth distributions, as demonstrated in works such as \cite{Masry2003}.

\subsection{The Case of Non-Standard Error}

In the deconvolution literature, it is typically assumed that the noise density \(r(x)\) is known and that its Fourier transform \(\phi_{r}(t)\) has no vanishing points, i.e.,
\[
\phi_{r}^{-1}(\{0\}) := \{t \in \mathbb{R} : \phi_{r}(t) = 0\} = \varnothing.
\]
However, many error densities violate this assumption. Examples include triangular densities, their convolutions with arbitrary densities, and uniform densities. These are informally referred to as non-standard errors.

In such cases, ridge-parameter regularization is employed to estimate \(\phi_{r}(t)\). For further details on ridge-parameter regularization in deconvolution problems, readers are directed to \cite{Hall&Meister} and related references. To the best of our knowledge, no existing results address the estimation of the hazard rate function from data contaminated by non-standard errors. Addressing this gap represents a valuable research opportunity.

Another interesting direction is the estimation of the hazard rate function when the error distribution is unknown but estimable. In this framework, we assume the availability of an additional sample \(\{e_i\}_{i=1}^n\) drawn from the error distribution \(e\), collected from an independent trial. The proposed approach involves two steps: first, estimating \(r(x)\) using \(\{e_i\}_{i=1}^n\), and second, estimating the density \(f(x)\) using classical deconvolution methods. The cumulative distribution function \(F_n\) can then be obtained via the blogging method. This approach is effective regardless of whether the density \(r(x)\) is non-standard or not.

By addressing these open questions, future research can significantly advance the field of hazard rate estimation in the presence of additive measurement errors.

\section{Simulation study}

This section is motivated by the pointwise simulation of the hazard rate estimator \(\lambda_{n}(\cdot)\) with the goal of evaluating its performance quality and implementing the corresponding asymptotic normality. First, we employ a kernel with the characteristic function  
\[
\phi_{k}(t) = (1 - t^{2})^{3} \cdot \mathbf{1}_{[-1,1]}(t),
\]  
to evaluate the estimator and determine the asymptotic values \(D_{1}\) and \(D_{2}\) (as defined in equations \eqref{D1} and \eqref{D2}). These values correspond to the variances and covariance, respectively. The kernel \(\phi_{k}\) satisfies all the necessary assumptions required for our results to hold. For a detailed expression of the kernel function \(k(x)\), readers can refer to \cite{Fan1992}.  

The simulations in this study were conducted using the R package \texttt{deconvolve} \cite{delaigle2021deconvolve}. Additionally, MATLAB codes for reproducing the results are available on Delaigle's webpage.

\subsection{Designing the Simulated Data}

We provide detailed results for all possible scenarios of the hazard rate function estimator, covering cases commonly encountered in practice: Non-Monotonic Hazard Rate (NMHR), Increasing Hazard Rate (IHR), Decreasing Hazard Rate (DHR), and Constant Hazard Rate (CHR).

We avoid implementing the processes described in Examples \ref{Example1} and \ref{Example2} because the true hazard rate, which is necessary for evaluating the performance of our estimator, is unknown. Moreover, Example \ref{Example1} is infeasible to simulate in practice, as it requires an infinite sequence \( \{\varepsilon_{i}\}_{i>1} \). The primary reason for excluding these examples is the lack of a known true hazard rate.

To ensure positive association (PA), we define \( X_j = \exp(Z_j) \), where
\[
Z_{j} = \frac{\varepsilon_{j-1} + \varepsilon_{j-2}}{2},
\]
and \( \{\varepsilon_j\}_{j=1}^{n} \) is an i.i.d. sequence drawn from a standard Gaussian process. In this case, the random variables \( \{X_{i}\}_{i=1}^{n} \) follow a \( \text{LogNormal}(0,1) \) distribution. The hazard rate function in this example is NMHR and remains relatively flat around its peak, making it particularly suitable for comparative analysis in this region.

Additionally, we conduct experiments to analyze other behaviors of the hazard rate function. For this, we employ the Weibull distribution, denoted as \( W(a, b) \), where \( a \) is the shape parameter and \( b \) is the scale parameter. Since the scale parameter does not affect the shape of the hazard rate, we fix \( b = 1 \). We have the following cases: - When \( a > 1 \), the Weibull distribution exhibits an IHR.
- When \( a < 1 \), it exhibits a DHR.
- When \( a = 1 \), it reduces to the exponential distribution, corresponding to a CHR.

\textbf{Summary of Our Scenarios}:
\begin{itemize}
	\item \textbf{Increasing Hazard Rate (IHR)}: \( X \sim W(1.5,1) \)
	\item \textbf{Decreasing Hazard Rate (DHR)}: \( X \sim W(0.75,1) \)
	\item \textbf{Constant Hazard Rate (CHR)}: \( X \sim W(1,1) \)
	\item \textbf{Non-Monotonic Hazard Rate (NMHR)}: \( X_j = \exp(Z_j) \), where \( Z_j \) is defined as above.
\end{itemize}

These scenarios encompass a broad range of practical applications, allowing us to thoroughly evaluate the performance of the estimator under various conditions.

Next, we generate the measurement errors \( e \) as an i.i.d. random sequence \( \{ e_i \}_{i=1}^{n} \) drawn from a Laplacian distribution. The Laplacian distribution is chosen because it belongs to the class of \textit{ordinary smooth errors}, which frequently arise in practical applications. Notably, when the errors \( \{ e_i \}_{i\geq 1} \) follow a Laplacian distribution, the associated \textit{smoothness parameter} is \( \beta = 2 \).  

The variance of the measurement errors is regulated by the noise-to-signal ratio (NSR), defined as $NSR = \frac{\sigma_e}{\sigma_X},$ where \( \sigma_e \) and \( \sigma_X \) denote the standard deviations of the measurement error and the underlying random variable \( X \), respectively. In this study, we examine three levels of contamination: \( NSR = 0.1 \), \( 0.2 \), and \( 0.5 \), corresponding to \textbf{10\%}, \textbf{20\%}, and \textbf{50\%} noise contamination, respectively.

Using the data \(\{ Y_i = X_i + e_i \}_{i=1}^{n}\), the hazard rate estimator \(\lambda_n(x)\) is computed by varying \(x\) over a grid of points in \(\Omega = \{ x \in [0:0.01:6] \}\). To evaluate the impact of sample size on the performance of the estimator, we consider \(n = 1000\), \(2000\), and \(5000\). As highlighted in \cite{Phuong2019}, estimation from contaminated data converges slowly to the target \(X\)-function. Therefore, a large sample size is required to achieve an estimator with good performance, which explains the exclusion of smaller sample sizes in our analysis.

To visually illustrate the influence of the sample size \( n \) and the NSR on the quality of fit of the estimator, we present side-by-side plots of \( \lambda_n(x) \) and the true hazard rate function \( \lambda(x) \) (corresponding to NSR = 0), as shown in Figure \ref{F_Vs_Fn}.  
\begin{figure}[]
	\centering
	\begin{subfigure}[b]{\linewidth}
		\centering
		\includegraphics[width=4.5cm,height=5cm]{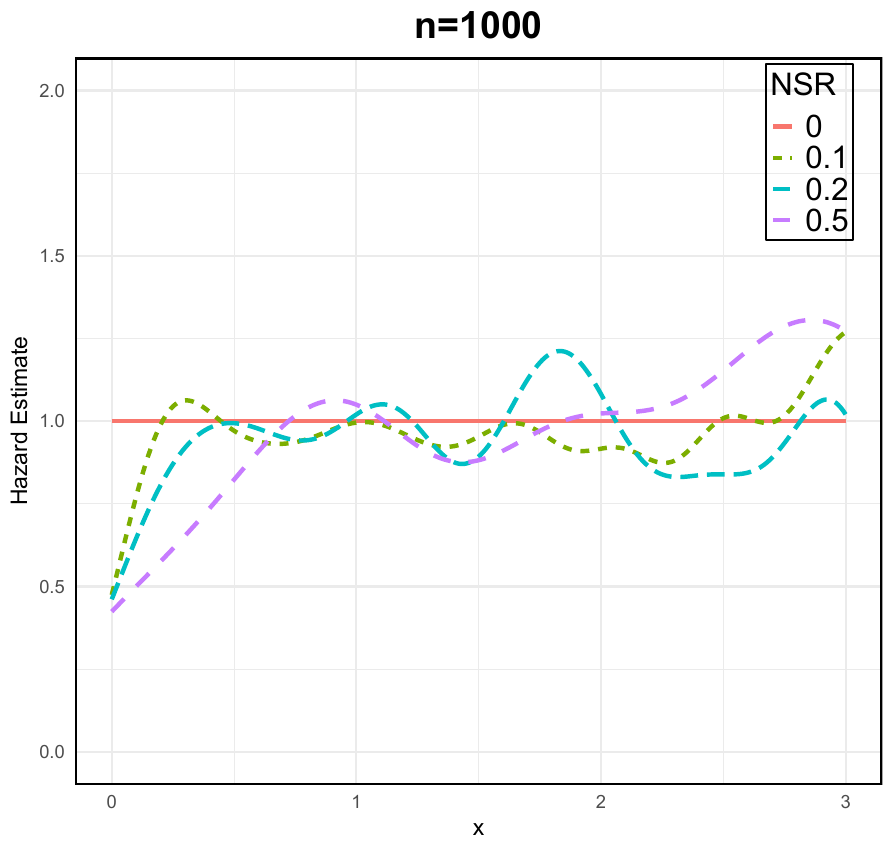}
		\includegraphics[width=4.5cm,height=5cm]{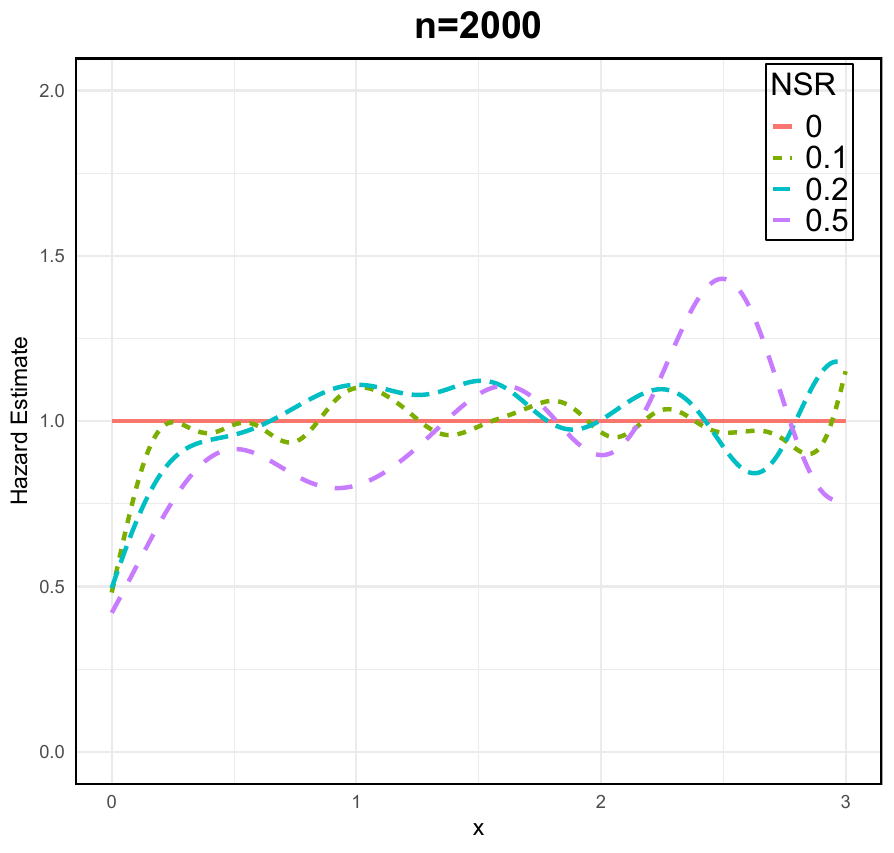}
		\includegraphics[width=4.5cm,height=5cm]{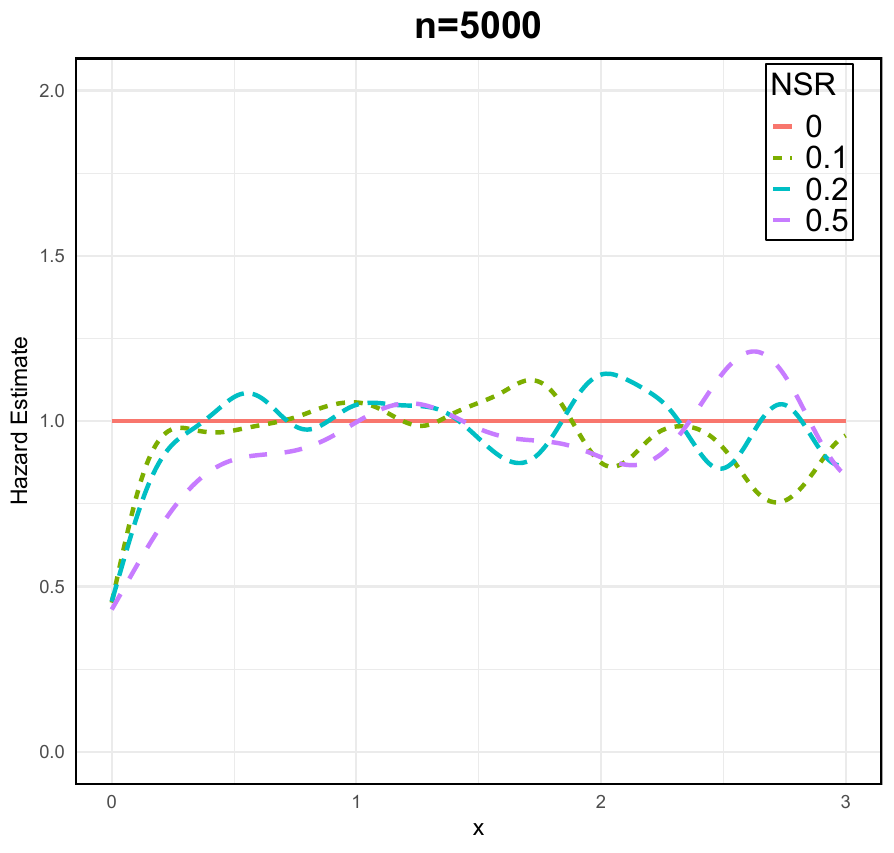}
		\caption{$W\left( 1 ,1\right) $}
	\end{subfigure}
	
	\begin{subfigure}[b]{\linewidth}
		\centering
		\includegraphics[width=4.5cm,height=5cm]{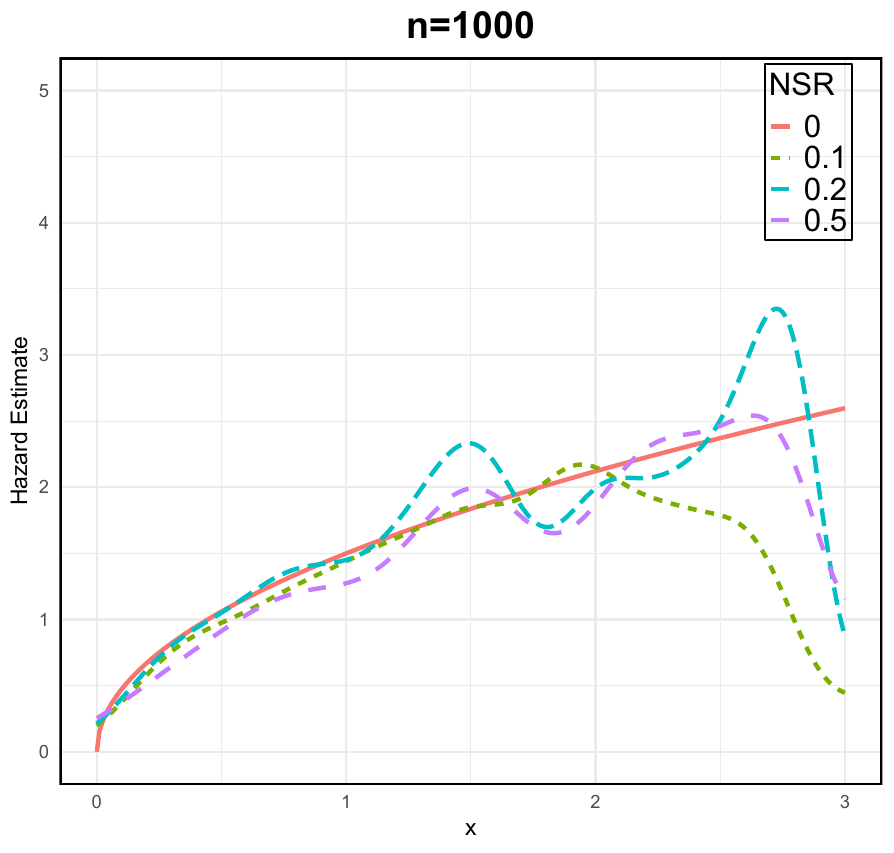}
		\includegraphics[width=4.5cm,height=5cm]{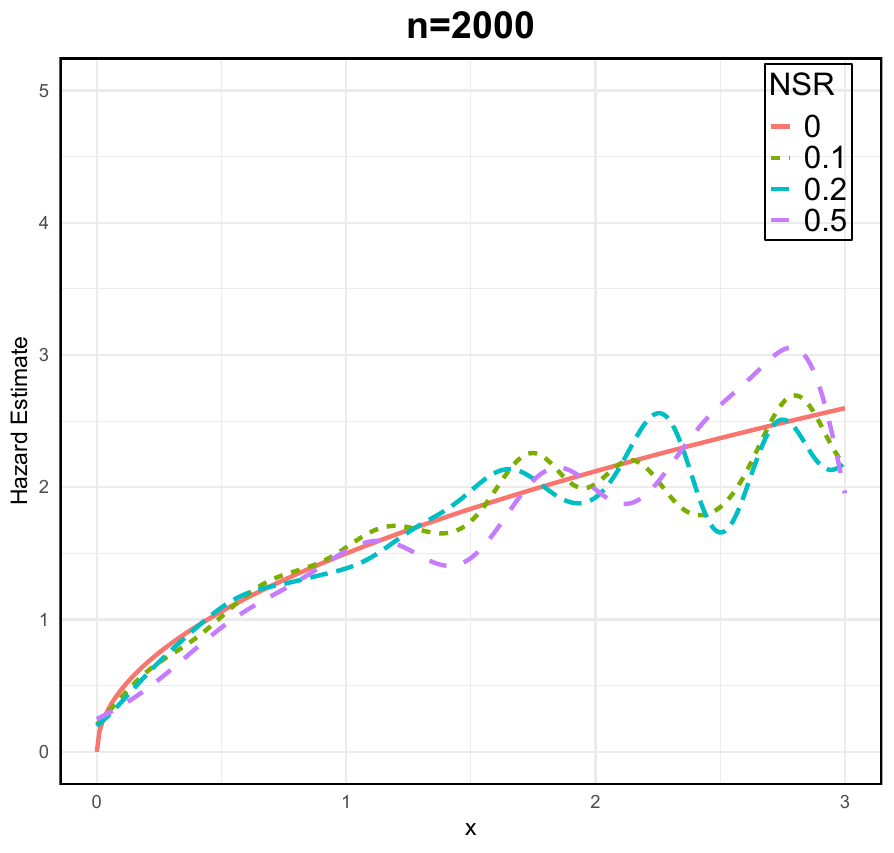}
		\includegraphics[width=4.5cm,height=5cm]{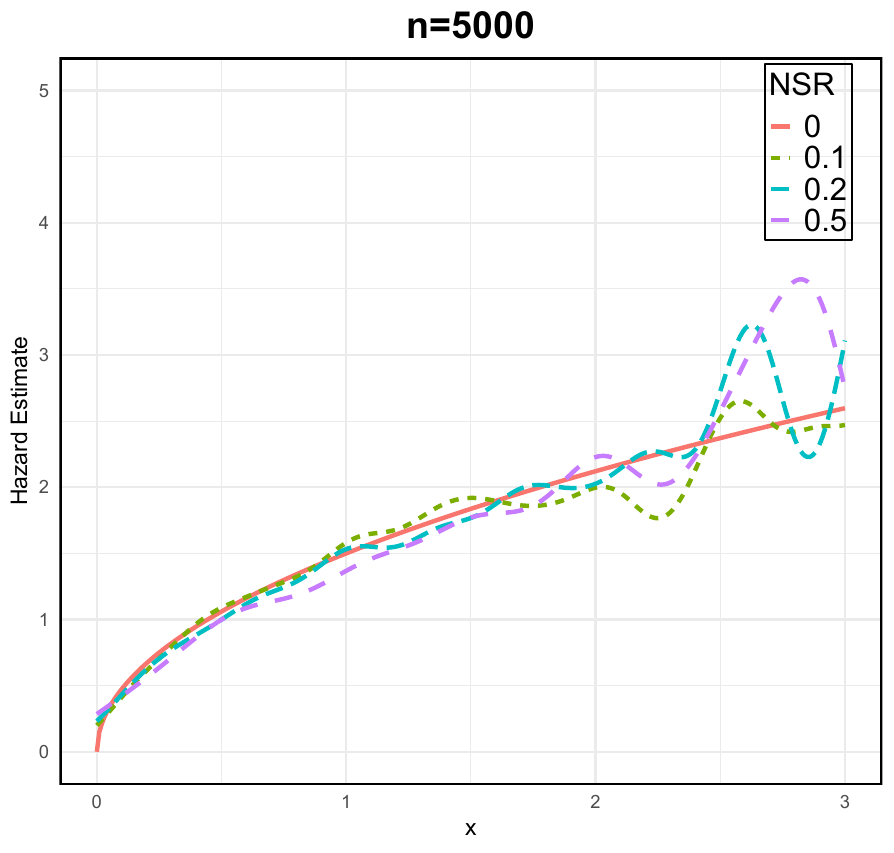}
		\caption{Weibull $W\left( 1.5 ,1\right) $}
	\end{subfigure}
	
	\begin{subfigure}[b]{\linewidth}
		\centering
		\includegraphics[width=4.5cm,height=5cm]{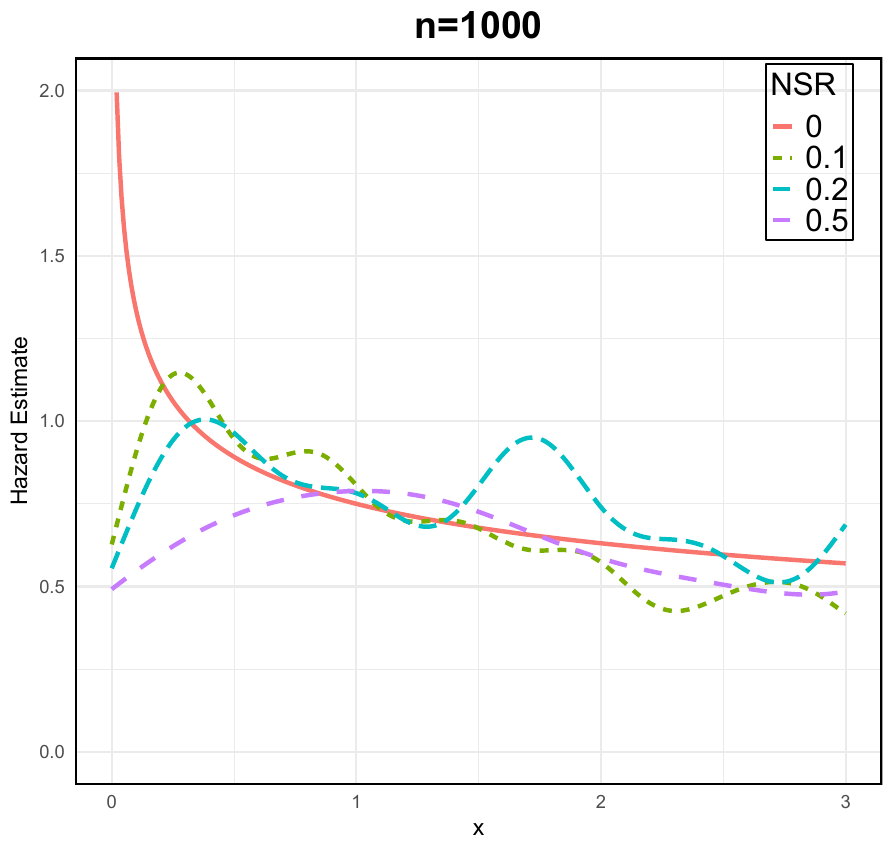}
		\includegraphics[width=4.5cm,height=5cm]{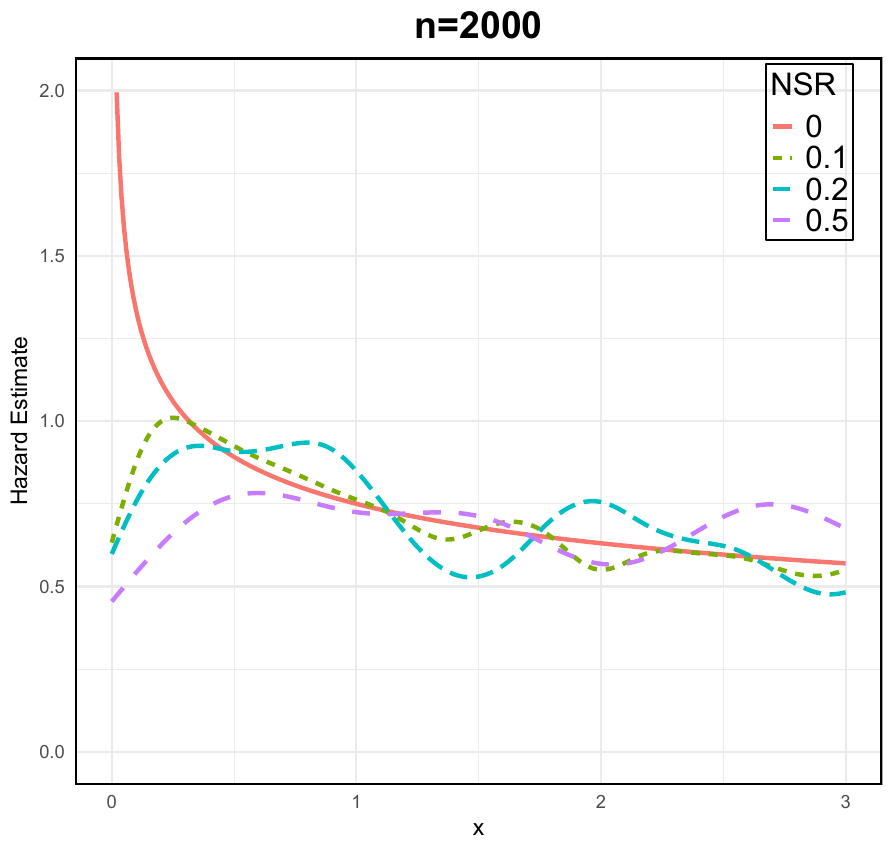}
		\includegraphics[width=4.5cm,height=5cm]{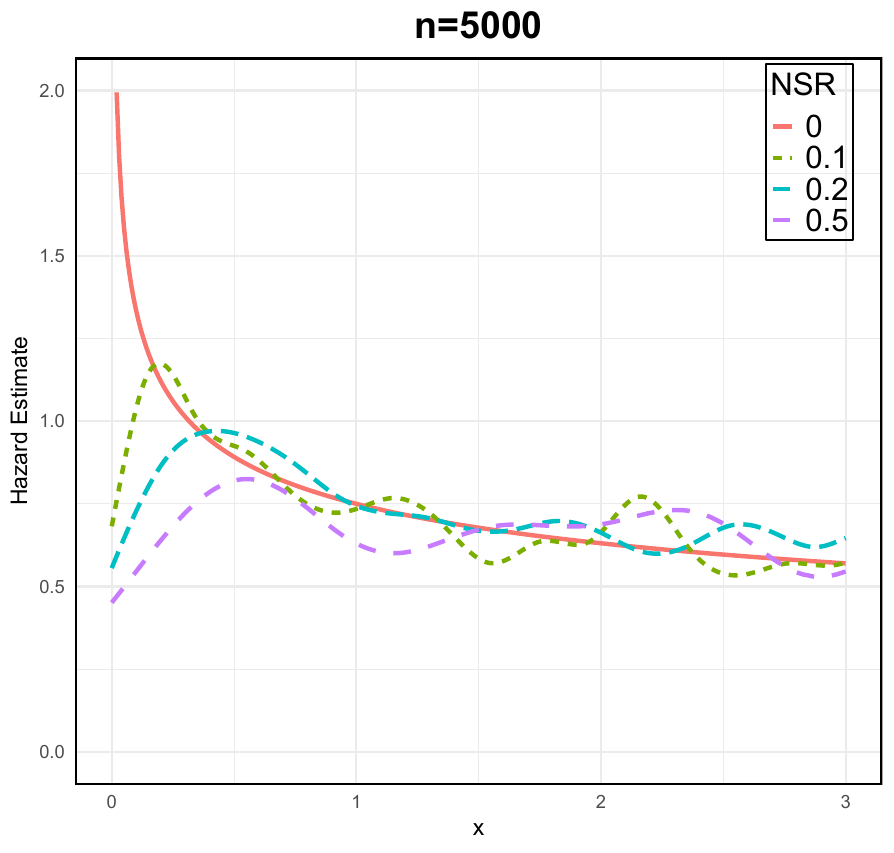}
		\caption{Weibull $W\left( 0.75 ,1\right) $}
	\end{subfigure}
	
	\begin{subfigure}[b]{\linewidth}
		\centering
		\includegraphics[width=4.5cm,height=5cm]{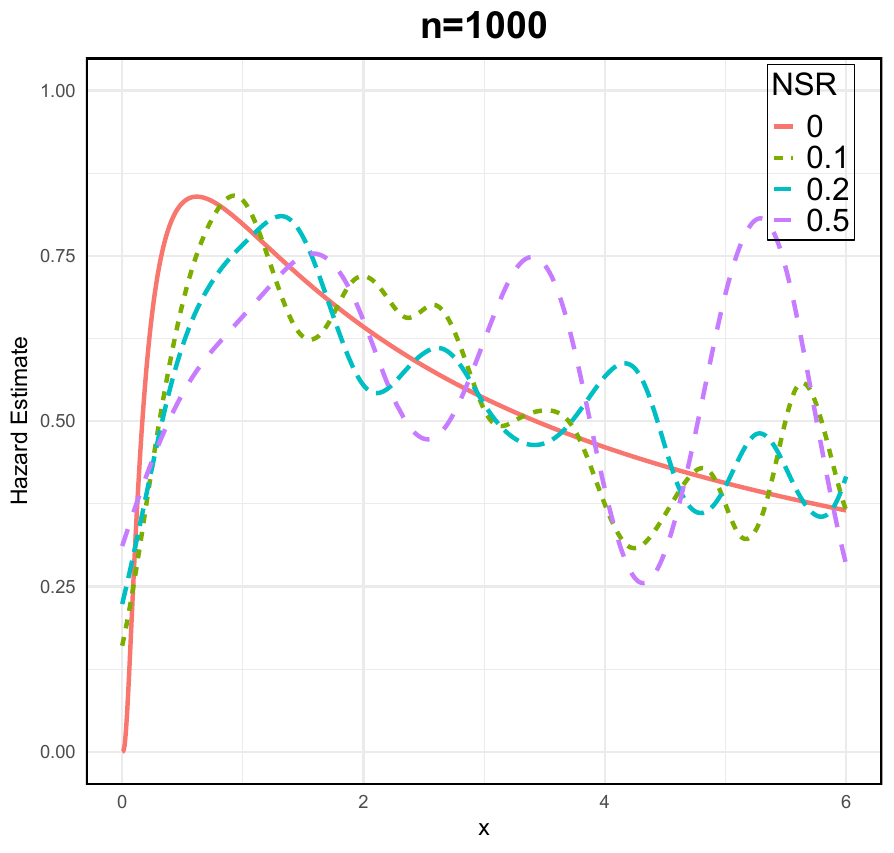}
		\includegraphics[width=4.5cm,height=5cm]{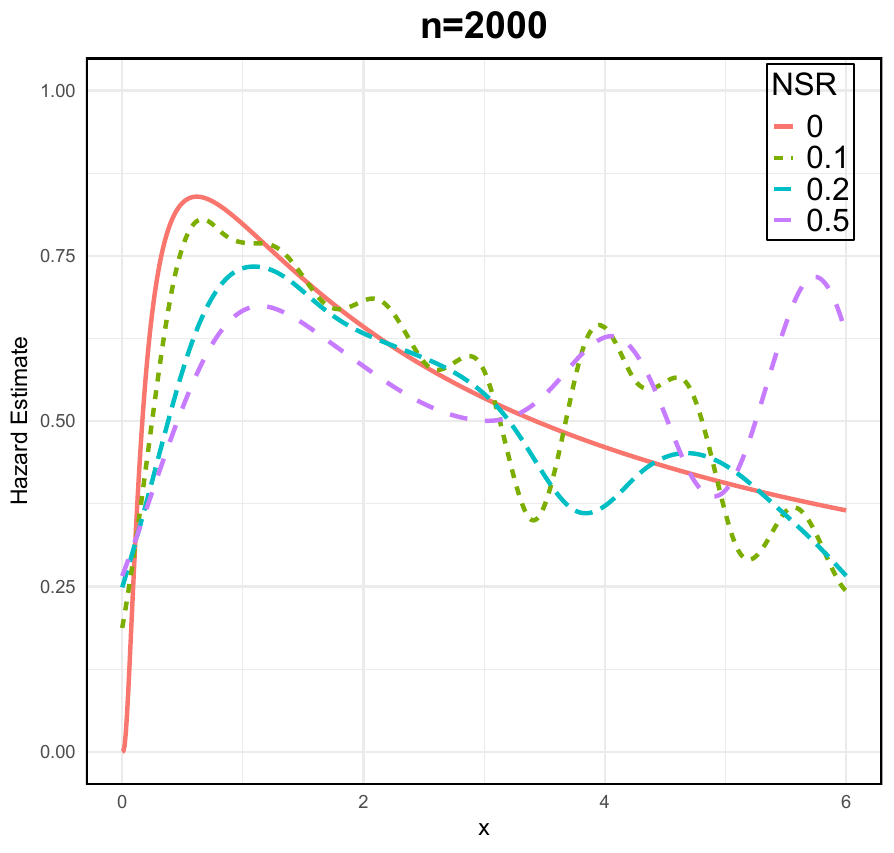}
		\includegraphics[width=4.5cm,height=5cm]{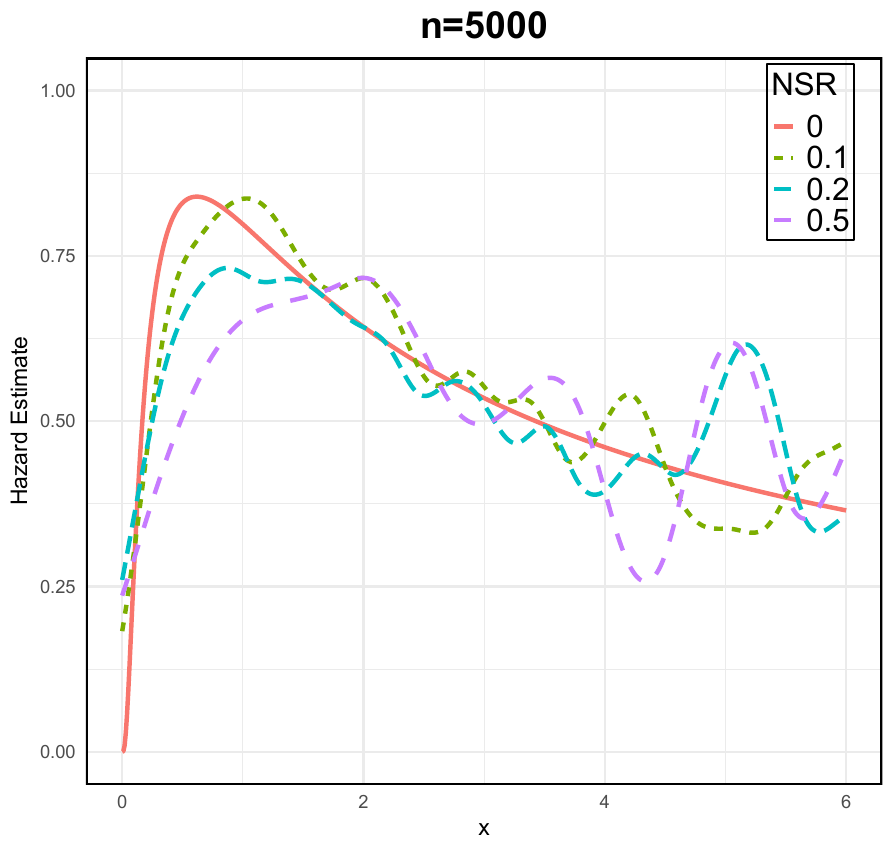}
		\caption{\( \text{LogNormal}(0,1) \)} % <- Check if this caption should be different
	\end{subfigure}
	
	\caption{ The true hazard rate function (corresponding to NSR = 0) is compared against its estimates obtained under different NSR.}
	\label{F_Vs_Fn}
\end{figure}

Summarizing the key findings from these plots:  
First, the hazard rate estimator demonstrates a strong ability to recover the true function across different scenarios. Next, the estimation performance remains acceptable even for \( \text{NSR} = 50\% \), and it significantly improves as the sample size \( n \) increases. In general, lower NSR values and larger sample sizes lead to better estimation accuracy.  

Additionally, we highlight that the quality of the estimation depends directly on the shape of the true hazard rate function. Specifically, NMHR and IHR tend to be estimated more accurately and rapidly compared to CHR and DHR.

\subsection{Estimation with Large Sample Sizes}

From the plots in Figure~\ref{F_Vs_Fn}, we observe that the estimator exhibits convergence for large values of \( n \), particularly when the Noise-to-Signal Ratio (NSR) is low. In this subsection, we extend our analysis to large sample sizes to investigate the asymptotic behavior of the estimator. Due to the computational complexity of the estimation process, substantial resources are required. To efficiently handle these computations, we implement the estimation procedure in MATLAB, leveraging the Fast Fourier Transform (FFT) for improved performance. 

Furthermore, we restrict our analysis to the NMHR case, as other scenarios yield similar results and lead to the same interpretation. The findings of this study are summarized in the plots presented in Figure~\ref{LS}.

\begin{figure}[h]
	\centering
	
	\begin{subfigure}[b]{0.3\textwidth}
		\centering
		\includegraphics[width=5cm,height=5.5cm]{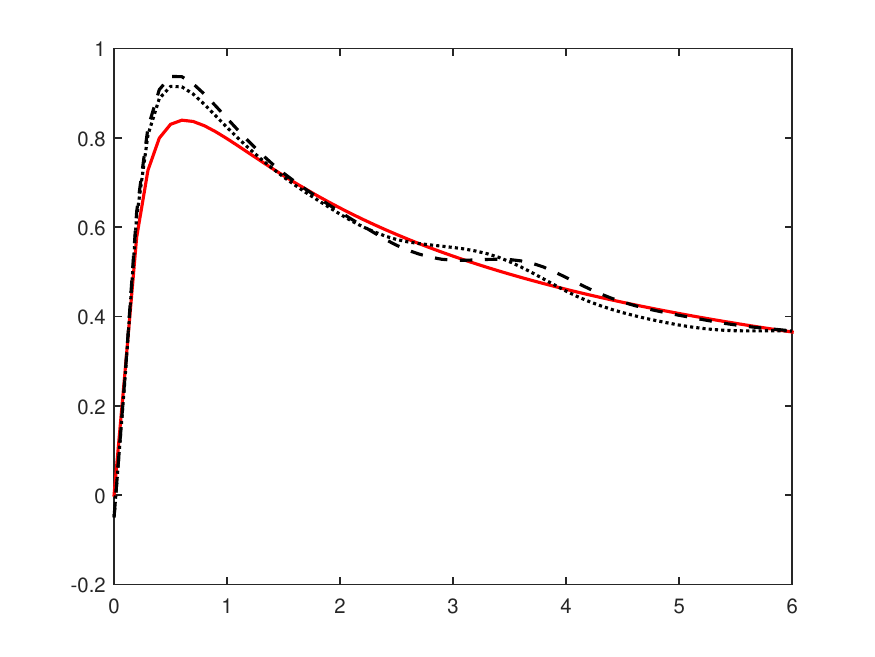}
		\caption{\( \text{NSR} = 0.1 \)}
	\end{subfigure}
	\hspace{0.5cm}
	\begin{subfigure}[b]{0.3\textwidth}
		\centering
		\includegraphics[width=5cm,height=5.5cm]{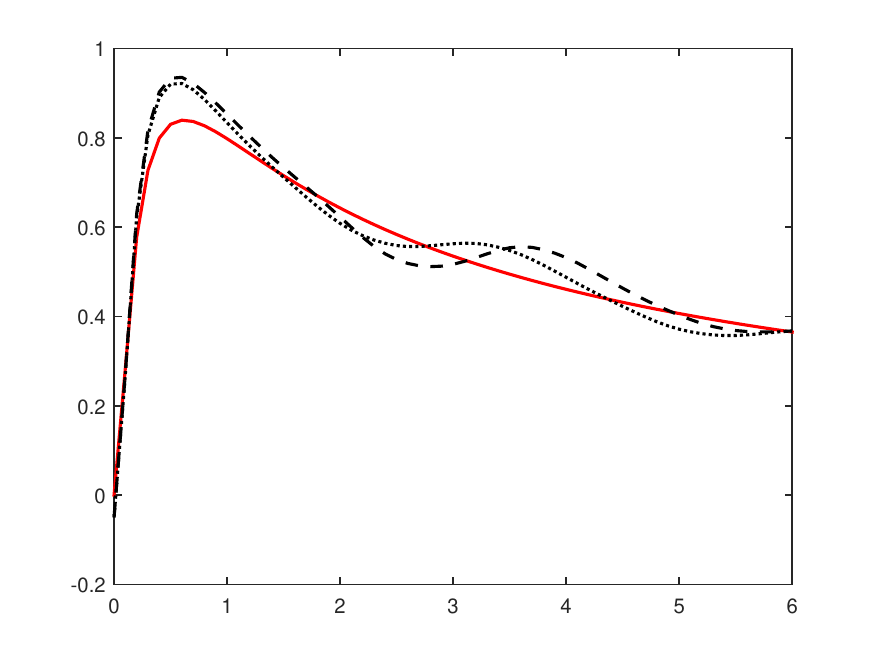}
		\caption{\( \text{NSR} = 0.2 \)}
	\end{subfigure}
	\hspace{0.5cm} 
	\begin{subfigure}[b]{0.3\textwidth}
		\centering
		\includegraphics[width=5cm,height=5.5cm]{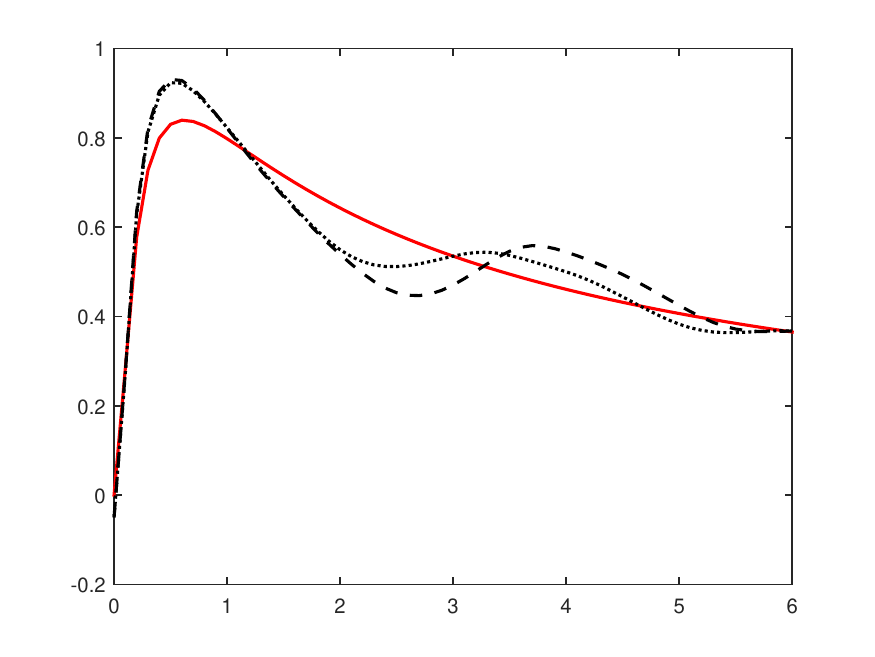}
		\caption{\( \text{NSR} = 0.5 \)}
	\end{subfigure}
	
	\caption{Hazard rate estimator computed from large sample size 
		observations contaminated by additive noise with \( \text{NSR} = 0.1 \), \( 0.2 \), and \( 0.5 \). 
		The solid line represents the true hazard rate function, the dashed line corresponds to 
		the hazard rate estimator for \( n = 10000 \), and the dotted line represents 
		the hazard rate estimator for \( n = 15000 \).}
	\label{LS}
\end{figure}

As shown in Figure~\ref{LS}, the estimation from samples with low contamination (\(\text{NSR} = 10\%\)) exhibits strong performance, improving progressively as \( n \) increases. Likewise, the quality of the fit enhances significantly with larger sample sizes, suggesting that the estimation error becomes negligible for sufficiently large \( n \). Additionally, we observe that the fit deteriorates as the NSR increases. However, for larger values of \( n \), the impact of NSR on the fit becomes less pronounced.

\subsection{Asymptotic Normality and Confidence Interval}

In this subsection, we investigate the asymptotic normality of the hazard rate estimator using normal probability plots, comparing the empirical distributions to the standard normal distribution. This analysis is conducted for Laplacian errors with \( \text{NSR} = 0.25 \) and \( M = 500 \) replications of samples of size \( n \), as illustrated in Figure~\ref{AN}.

Additionally, we construct confidence intervals for the hazard rate based on \( \lambda(0.5) \) in the case of NMHR. To achieve this, we generate \( M = 1000 \) replications of samples of size \( n \) (\( n = 1000 \), \( 2000 \), and \( 5000 \)) under the contamination scenario described earlier. The coverage probabilities (\( CP \)) and average lengths (\( AL \)) of these confidence intervals are summarized in Table~\ref{tab2}.
\begin{table}[]
	\caption{The coverage probabilities and average lengths
		of $95\%$ confidence intervals of $\protect\lambda \left( 0.5\right) $ .}
	\label{tab2}%
	\begin{center}
		\begin{tabular}{lrlll}
			\hline
			& \multicolumn{1}{l}{%
				\begin{tabular}{@{}c}
					$NSR$ \\ 
					$n$%
				\end{tabular}%
			} & 
			\begin{tabular}{c}
				0.1 \\ \hline
				$AL$ \qquad $CP$%
			\end{tabular}
			& 
			\begin{tabular}{c}
				0.25 \\ \hline
				$AL$ \qquad $CP$%
			\end{tabular}
			& 
			\begin{tabular}{c}
				0.5 \\ \hline
				$AL$ \qquad $CP$%
			\end{tabular}
			\\ \hline
			& 1000 & 
			\begin{tabular}{cl}
				0.0781 & 0.941%
			\end{tabular}
			& 
			\begin{tabular}{cl}
				0.0836 & 0.938%
			\end{tabular}
			& 
			\begin{tabular}{cl}
				0.0871 & 0.931%
			\end{tabular}
			\\ 
			& 2000 & 
			\begin{tabular}{cl}
				0.0561 & 0.945%
			\end{tabular}
			& 
			\begin{tabular}{cl}
				0.0578 & 0.941%
			\end{tabular}
			& 
			\begin{tabular}{cl}
				0.0612 & 0.937%
			\end{tabular}
			\\ 
			&5000 & 
			\begin{tabular}{cl}
				0.0521 & 0.958%
			\end{tabular}
			& 
			\begin{tabular}{cl}
				0.0571 & 0.945%
			\end{tabular}
			& 
			\begin{tabular}{cl}
				0.0688 & 0.946%
			\end{tabular}
			\\ \hline
		\end{tabular}%
	\end{center}
\end{table}

\begin{figure}[h]
	\centering
	
	\begin{subfigure}[b]{0.3\textwidth}
		\centering
		\includegraphics[width=4.5cm,height=5.5cm]{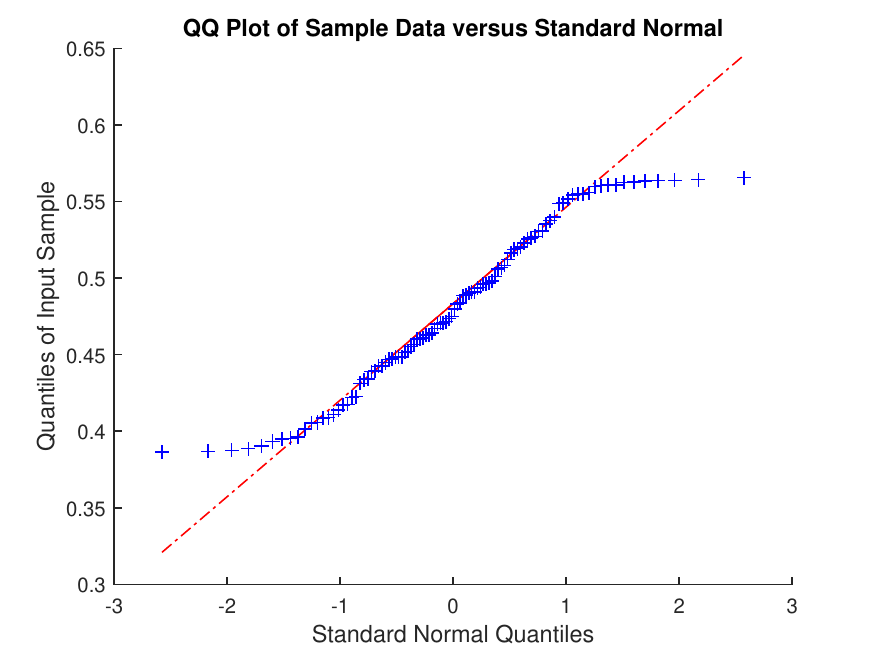}
		\caption{\( n = 1000 \)}
		\label{fig:n1000}
	\end{subfigure}
	\hspace{0.3cm}  
	\begin{subfigure}[b]{0.3\textwidth}
		\centering
		\includegraphics[width=4.5cm,height=5.5cm]{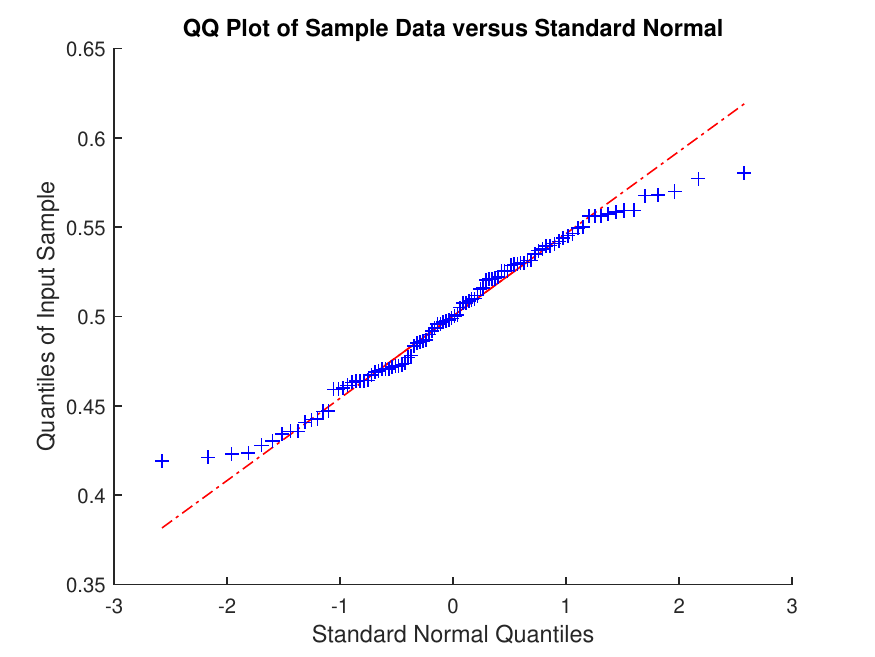}
		\caption{\( n = 2000 \)}
		\label{fig:n2000}
	\end{subfigure}
	\hspace{0.3cm}  
	\begin{subfigure}[b]{0.3\textwidth}
		\centering
		\includegraphics[width=4.5cm,height=5.5cm]{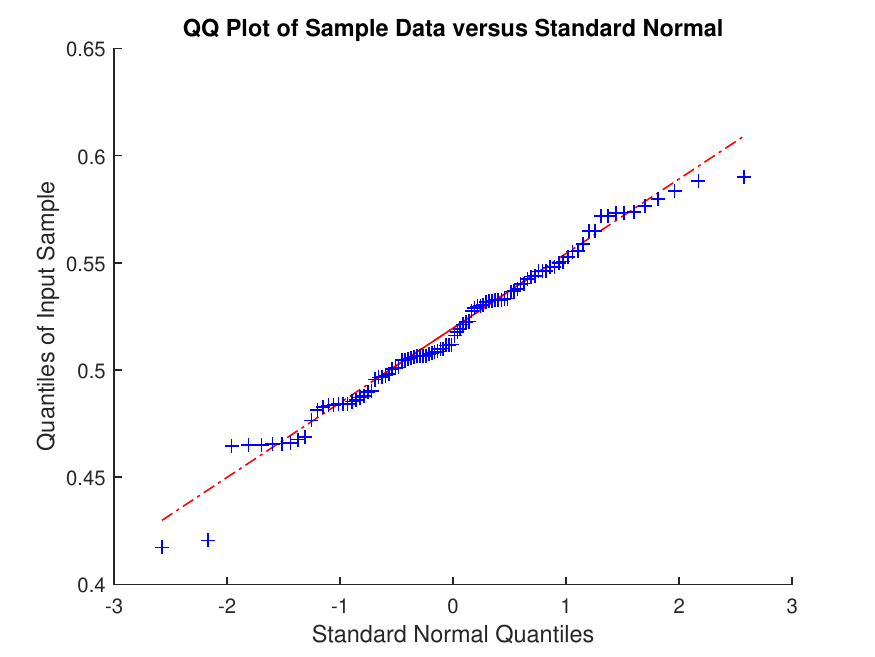}
		\caption{\( n = 5000 \)}
		\label{fig:n5000}
	\end{subfigure}
	
	\caption{Normal-probability plots of \( \lambda _{n}(0.5) \) based on Laplacian errors with \( \text{NSR} = 0.1 \).}
	\label{AN}
\end{figure}

As to the asymptotic normality, Figure \ref{AN}
shows that the sampling distribution of the hazard rate estimator matches
the Gaussian distribution. This match increases and becomes better along
with $n$. And for confidence intervals, Table \ref{tab2} confirms that the
coverage probabilities ($CP$) rise together with the sample size $n$. On the
other hand, we notice that the average lengths ($AN$) decrease reversely
with the sample size $n$.

\section{Proofs and Auxiliary Results}

To evaluate the precise asymptotic expression of $\text{Var}(F_{n}(x))$ in Theorem %
\ref{variance}, one shall establish an approximation of the identity. To do
so, we work analogously as in Lemma 1 of \cite{Masry2003}. Lemma \ref{ident}
hereafter is in order.

\begin{lemma}
	\label{ident} Under conditions (\textbf{H1}), (\textbf{H3}) and (\textbf{H4}), for $l=1,2$, we have
	
	\[
	h_{n}^{l\beta }E\left[ \left( M_{h_{n},1}\left( x\right) \right) ^{l}\right]
	\rightarrow G(x)D_{1}^{l},
	\]%
	where $D_{1}$ is defined in (\ref{D1}).
\end{lemma}

\begin{proof}[Proof of Lemma \protect\ref{ident}]

	We split the integral as follows:
	\[
	\begin{aligned}
		\lim_{n \rightarrow +\infty} h_n^\beta W_{h_n}(u) 
		&= \lim_{n \rightarrow +\infty} \frac{h_n^\beta}{2\pi} 
		\int_{-\infty}^{+\infty} \cos(tu) \frac{\phi_k(t)}{\phi_r(t/h_n)} \, dt \\
		&\quad + \lim_{n \rightarrow +\infty} \frac{h_n^\beta}{2\pi} i 
		\int_{-\infty}^{+\infty} \sin(tu) \frac{\phi_k(t)}{\phi_r(t/h_n)} \, dt \\
		&=: l_1 + l_2.
	\end{aligned}
	\]
	
	First, using \textbf{(H1)}-2) together with the dominated convergence theorem, we obtain:
	\[
	\begin{aligned}
		l_1 &= \lim_{n \rightarrow +\infty} \frac{1}{2\pi} 
		\int_{-\infty}^{+\infty} \cos(tu) \frac{t^\beta \phi_k(t)}{(t/h_n)^\beta \phi_r(t/h_n)} \, dt \\
		&= \frac{1}{2\pi \beta_1} 
		\int_{-\infty}^{+\infty} \cos(tu) t^\beta \phi_k(t) \, dt.
	\end{aligned}
	\]
	
	Since \(\phi_k\) is an even function (as stated in \textbf{(H2)}) and \(\beta\) is an even integer, the integrand \(\cos(tu) t^\beta \phi_k(t)\) is also an even function. Therefore:
	\[
	l_1 = L(u),
	\]
	where \(L(u)\) is defined in \(\eqref{L}\). By similar steps, we obtain:
	\[
	l_{2} = \frac{i}{2\pi \beta_{1}} \int_{-\infty}^{+\infty} t^{\beta} \sin(tu) \phi_k(t) \, dt.
	\]
	
	The term \( l_{2} \) vanishes because the integrand \( t^{\beta} \sin(tu) \phi_k(t) \) is an odd function. Therefore, we can conclude that:
	\begin{equation}\label{Result1}
		\lim_{n \rightarrow +\infty} h_n^\beta W_{h_n}(u) = L(u).
	\end{equation}
	
	Now, for \( l = 1 \), it is easy to see that  
	\[
	h_{n}^{\beta} E\left[ M_{h_{n}, 1}(x) \right] = \int_{-\infty}^{+\infty} \left( \int_{-\infty}^{(x - s) / h_n} h_{n}^{\beta} W_{h_n}(u) \, du \right) g(s) \, ds.
	\]
	
	Next, we split the interval of integration appropriately as follows:  
	\begin{eqnarray}
		h_{n}^{\beta} E\left[ M_{h_{n}, 1}(x) \right] &=& \int_{-\infty}^{x} \left( \int_{-\infty}^{(x - s) / h_n} h_{n}^{\beta} W_{h_n}(u) \, du \right) g(s) \, ds \nonumber \\
		&& + \int_{x}^{+\infty} \left( \int_{-\infty}^{(x - s) / h_n} h_{n}^{\beta} W_{h_n}(u) \, du \right) g(s) \, ds \nonumber \\
		&=& I_{n, 1} + I_{n, 2}.
	\end{eqnarray}
	
	Now, consider \( I_{n, 2} \). Since we are integrating with respect to \( s \) from \( x \) to \( +\infty \), it follows that \( x - s \leq 0 \) in this range. Therefore,  
	\[
	\frac{x - s}{h_n} \to -\infty \quad \text{as} \quad n \to +\infty.
	\]
	From Remark \ref{pdf}, we observe that  
	$\left| \int_{-\infty}^{(x - s) / h_n} W_{h_n}(u) \, du \right| \leq 1.$  
	Moreover, since $\int_{x}^{+\infty} g(s) \, ds = 1 - G(x) < \infty,$ we can apply the dominated convergence theorem and obtain the following result:
	
	\[
	\lim_{n \to \infty} I_{n, 2} = \int_{x}^{+\infty} \left( \lim_{n \to \infty} \int_{-\infty}^{(x - s) / h_n} h_{n}^{\beta} W_{h_n}(u) \, du \right) g(s) \, ds \to 0 \quad \text{as} \quad n \to +\infty.
	\]

	For \( I_{n,1} \), we split the interval of integration as follows:
	\[
	\begin{aligned}
		I_{n,1} &= \int_{-\infty}^{x} \left( \int_{-\infty}^{-(x-s)/h_n} h_n^\beta W_{h_n}(u) \, du \right) g(s) \, ds + \int_{-\infty}^{x} \left( \int_{-(x-s)/h_n}^{(x-s)/h_n} h_n^\beta W_{h_n}(u) \, du \right) g(s) \, ds \label{e3} \\
		&=: I_{n,1}^{\prime} + I_{n,1}^{\prime \prime}.
	\end{aligned}
	\]
	
	Using a similar argument, we observe that \( -(x-s)/h_n \rightarrow -\infty \) as \( n \rightarrow +\infty \). Thus:
	\[
	I_{n,1}^{\prime} \rightarrow 0 \quad \text{as } n \rightarrow +\infty. \label{e4}
	\]
	
	The main task is to calculate the asymptotic value of \( I_{n,1}^{\prime \prime} \). Applying the dominated convergence theorem and using the result from \eqref{Result1}, we obtain:
	\[
	\begin{aligned}
		\lim_{n \rightarrow +\infty} I_{n,1}^{\prime \prime} &= \int_{-\infty}^{x} \left( \int_{-\infty}^{+\infty} L(u) \, du \right) g(s) \, ds \\
		&= \left( \int_{-\infty}^{+\infty} L(u) \, du \right) G(x).
	\end{aligned}
	\]
	
	The desired result is established for \( l = 1 \). Next, for \( l = 2 \), we proceed as follows:

	\begin{eqnarray*}
		h_{n}^{2\beta }E\left[ M_{h_{n},1}^{2}\left( x\right) \right]
		&=&\int\limits_{-\infty }^{x}\left( \int\limits_{-\infty }^{\left(
			x-s\right) /h_{n}}{}h_{n}^{\beta }W_{h_{n}}\left( u\right) du\right)
		^{2}g\left( s\right) ds+\int\limits_{x}^{+\infty }\left(
		\int\limits_{-\infty }^{\left( x-s\right) /h_{n}}W_{h_{n}}\left( u\right)
		du\right) ^{2}g\left( s\right) ds \\
		&=&I_{1}+I_{2}.
	\end{eqnarray*}
	
	We follow similar arguments to see that $I_{1}\rightarrow G\left( x\right) 
	D_{1}^{2}$ and $I_{2}\rightarrow 0$ as $n\rightarrow +\infty .$
\end{proof}

The next lemma which due to \cite{Birkel1988} plays a crucial role to
evaluate the covariance when dealing with associated (PA) data.

\begin{lemma}[Birkel]
	\label{Berkel} Suppose that $\left\{ Y_{j}\right\} _{j\in I}$ is a finite
	collection of positively associated r.v.'s. Let $A$ and $B$ be subsets of 
	$I$ and $\Phi _{1}(.)$, $\Phi  _{2}(.)$ be functions on $ 
	%TCIMACRO{\U{211d} }%
	%BeginExpansion
	\mathbb{R}
	%EndExpansion
	^{\left\vert A\right\vert }$ and $ 
	%TCIMACRO{\U{211d} }%
	%BeginExpansion
	\mathbb{R}
	%EndExpansion
	^{\left\vert B\right\vert }$  respectively, with bounded first order partial
	derivatives, then we have  
	\begin{equation*}
		\left\vert \text{Cov}\left[ \Phi _{1}(Y_{i},i\in A),\Phi _{2}(Y_{j},j\in B) \right]
		\right\vert \leq \sum_{i\in A}\sum_{j\in B}\left\Vert \frac{ \partial \Phi
			_{1}}{\partial t_{i}}\right\Vert _{\infty }.\left\Vert \frac{ \partial \Phi
			_{2}}{\partial t_{j}}\right\Vert _{\infty }\text{Cov}(Y_{i},Y_{j}),
	\end{equation*}
	
	where $\left\Vert .\right\Vert _{\infty }$ stands for the sup-norm, i.e. $%
	\left\Vert \frac{\partial \Phi _{i}}{ \partial t_{l}}\right\Vert _{\infty
	}=\max \left\{ \left\Vert \frac{\partial ^{+}\Phi _{i}}{\partial t_{l}}%
	\right\Vert _{\infty },\left\Vert \frac{ \partial ^{-}\Phi _{i}}{\partial
		t_{l}}\right\Vert _{\infty }\right\} $.
\end{lemma}

\begin{lemma}
	\label{prop2} Under hypothesis (\textbf{H1}) and for $n$ large enough, in
	addition to:
	
	1) If (\textbf{H3})-2 holds, then we have 
	\[
	h_{n}^{\beta }\left\Vert W_{h_{n}}\right\Vert _{\infty }\leq C. 
	\]
	
	2) If (\textbf{H3})-1 holds, then $W_{h_{n}}$ is continuous
	and first-order derivative with 
	\[
	h_{n}^{\beta }\left\Vert W_{h_{n}}^{\prime }\right\Vert _{\infty }\leq C,
	\]
	
	3) If (\textbf{H3})-3 holds true, then we have 
	\[
	h_{n}^{\beta }\left\Vert M_{h_{n}}\right\Vert _{\infty }\leq C.
	\]
\end{lemma}

\begin{proof}[Proof of Lemma \protect\ref{prop2}]
	To establish the first assertion, we begin by noting that  
	
	\[
	\left\Vert W_{h_{n}}\right\Vert _{\infty } \leq \frac{1}{2\pi } \int\limits_{-\infty }^{+\infty } \left\vert \frac{\phi _{k}\left( t\right) }{\phi _{r}\left( t/h_{n}\right) } \right\vert dt.
	\]
	
	Applying assumption (\textbf{H1})-2, we obtain  
	
	\begin{eqnarray*}
		\lim_{n \to +\infty} h_{n}^{\beta } \int\limits_{-\infty }^{+\infty } \left\vert \frac{\phi _{k}\left( t\right) }{\phi _{r}\left( t/h_{n}\right) } \right\vert dt
		&=& \lim_{n \to +\infty} \int\limits_{-\infty }^{+\infty } \left\vert \frac{t^{\beta } \phi _{k}\left( t\right) }{\left( t/h_{n}\right) ^{\beta } \phi _{r}\left( t/h_{n}\right) } \right\vert dt \\ 
		&=& \frac{1}{\beta _{1}} \int\limits_{-\infty }^{+\infty } \left\vert t^{\beta} \phi _{k}\left( t\right) \right\vert dt.
	\end{eqnarray*}
	
	Thus, the desired result follows from assumption (\textbf{H3})-2.
	
	Next, we analyze the derivative of $ W_{h_n}(x) $, given by  
	
	\[
	W_{h_{n}}^{\prime }(x) = \frac{1}{2\pi } \int_{-\infty }^{+\infty } -i e^{-i tx} \frac{t \phi _{k}(t)}{\phi _{r}(t/h_{n})} dt.
	\]
	
	This leads to the bound  
	
	\[
	\left\Vert W_{h_{n}}^{\prime } \right\Vert _{\infty } \leq \frac{1}{2\pi } \int\limits_{-\infty }^{+\infty } \left\vert \frac{t \phi _{k}\left( t\right) }{\phi _{r}\left( t/h_{n}\right) } \right\vert dt.
	\]
	
	Applying the same reasoning and assumption (\textbf{H3})-1, we obtain the required result.
	
	Finally, using Fubini's theorem, we express $ M_{h_n}(-x) $ as  
	
	\[
	M_{h_{n}}(-x) = \frac{1}{2\pi } \int_{-\infty }^{+\infty } \frac{-e^{-i tx}}{i} \frac{\phi _{k}(t)}{t \phi _{r}(t/h_{n})} dt.
	\]
	
	From this, we derive the bound  
	
	\[
	\left\Vert M_{h_{n}} \right\Vert _{\infty } \leq \frac{1}{2\pi } \int\limits_{-\infty }^{+\infty } \left\vert \frac{\phi _{k}\left( t\right) }{t \phi _{r}\left( t/h_{n}\right) } \right\vert dt.
	\]
	
	Following the same steps as in the first assertion and applying assumption (\textbf{H3})-3, the result follows.

\end{proof}

\begin{proof}[Proof of Theorem \protect\ref{variance}]
	First, let us consider 
	\begin{equation}
		\tilde{M}_{h_{n},i}\left( x\right) =M_{h_{n},i}\left( x\right) -E\left[
		M_{h_{n},i}\left( x\right) \right]  \label{tild1}
	\end{equation}
	
	By simple calculation and using stationarity we can see that 
	\begin{eqnarray*}
		\text{Var}\left( 1-F_{n}\left( x\right) \right) &=&\text{Var}\left( F_{n}\left( x\right)
		\right) \\
		&=&\frac{1}{n}\text{Var}\left( \tilde{M}_{h_{n},1}\left( x\right) \right) +\frac{2}{
			n}\sum\limits_{i=2}^{n}\left( 1-\frac{i}{n}\right) \text{Cov}\left( \tilde{M}
		_{h_{n},1}\left( x\right) ,\tilde{M}_{h_{n},i}\left( x\right) \right) .
	\end{eqnarray*}
	
	Now, we are in position to prove that
	
	\begin{equation}  \label{item1}
		\lim_{n\rightarrow +\infty }h_{n}^{2\beta }\text{Var}\left( \tilde{M}
		_{h_{n},1}\left( x\right) \right) =G\left( x\right) \left[ 1-G\left(
		x\right) \right] D_{1}^2,
	\end{equation}
	and 
	\begin{equation}  \label{item2}
		h_{n}^{2\beta }\sum\limits_{i=2}^{n}\left\vert \text{Cov}\left( \tilde{M}
		_{h_{n},1}\left( x\right) ,\tilde{M}_{h_{n},i}\left( x\right) \right)
		\right\vert =o\left( 1\right) .
	\end{equation}
	
	Concerning point (\ref{item1}), we write 
	\[
	\text{Var}\left( \tilde{M}_{h_{n},1}\left( x\right) \right) =E\left[ \left(
	M_{h_{n},1}\left( x\right) \right) ^{2}\right] -\left( E\left[
	M_{h_{n},1}\left( x\right) \right] \right) ^{2}.
	\]
	
	Consequently 
	\[
	h_{n}^{2\beta }\text{Var}\left( \tilde{M}_{h_{n},1}\left( x\right) \right)
	=h_{n}^{2\beta }E\left[ \left( M_{h_{n},1}\left( x\right) \right) ^{2}\right]
	-\left( h_{n}^{\beta }E\left[ M_{h_{n},1}\left( x\right) \right] \right)
	^{2}.
	\]
	
	The target result followed directly using Lemma \ref{ident}.
	
	As to assertion (\ref{item2}), we deal as follows
	\[
	\sum\limits_{i=2}^{n}\left\vert \text{Cov}\left( \tilde{M}_{h_{n},1}\left( x\right)
	,\tilde{M}_{h_{n},i}\left( x\right) \right) \right\vert
	=\sum\limits_{i=2}^{\rho _{n}}\left\vert \text{Cov}\left( \tilde{M}_{h_{n},1}\left(
	x\right) ,\tilde{M}_{h_{n},i}\left( x\right) \right) \right\vert
	+\sum\limits_{i=\rho _{n}+1}^{n}\left\vert \text{Cov}\left( \tilde{M}
	_{h_{n},1}\left( x\right) ,\tilde{M}_{h_{n},i}\left( x\right) \right)
	\right\vert , 
	\]
	
	for some positive sequence $\left\{ \rho _{n}\right\} $ satisfies $\rho
	_{n}\rightarrow +\infty $ and $h_{n}^{2}\rho _{n}\rightarrow 0$ as $%
	n\rightarrow +\infty .$
	
	For of the first contribution, we have 
	\[
	\left\vert \text{Cov}\left( \tilde{M}_{h_{n},1}\left( x\right) ,\tilde{M}
	_{h_{n},i}\left( x\right) \right) \right\vert =\int_{-\infty }^{+\infty
	}\int_{-\infty }^{+\infty }M_{h_{n}}\left( \frac{x-u}{h_{n}}\right)
	M_{h_{n}}\left( \frac{x-v}{h_{n}}\right) \left[ g_{Y_{1},Y_{i}}\left(
	u,v\right) -g\left( u\right) g\left( v\right) \right] dudv. 
	\]
	
	From the assumption that $\left| g_{Y_{1},Y_{i}}\left( u,v\right) - g\left( u\right) g\left( v\right) \right| \leq M < +\infty$, we can have: 
	\[
	\left\vert \text{Cov}\left( \tilde{M}_{h_{n},1}\left( x\right) ,\tilde{M}%
	_{h_{n},i}\left( x\right) \right) \right\vert \leq Ch_{n}^{2}\left\Vert
	M_{h_{n}}\right\Vert _{\infty }^{2}. 
	\]
	
	From Lemma \ref{prop2}, we find 
	\[
	\left\vert \text{Cov}\left( \tilde{M}_{h_{n},1}\left( x\right) ,\tilde{M}%
	_{h_{n},i}\left( x\right) \right) \right\vert \leq \frac{C}{h_{n}^{2\beta -2}%
	}.
	\]
	
	This means that 
	\[
	h_{n}^{2\beta}\sum\limits_{i=2}^{\rho _{n}}\left\vert \text{Cov}\left( \tilde{M}_{h_{n},1}\left(
	x\right) ,\tilde{M}_{h_{n},i}\left( x\right) \right) \right\vert =O\left( 
	\rho _{n}h_n^2\right) . 
	\]
	
	Using the fact that $h_{n}^{2}\rho _{n}\rightarrow 0$ as $n\rightarrow
	+\infty $, gives 
	\[
	h_{n}^{2\beta}\sum\limits_{i=2}^{\rho _{n}}\left\vert \text{Cov}\left( \tilde{M}_{h_{n},1}\left(
	x\right) ,\tilde{M}_{h_{n},i}\left( x\right) \right) \right\vert =o\left(1\right) . 
	\]
	
	Concerning the second contribution (for $\rho _{n}+1\leq i\leq n$), we make
	use of Lemma \ref{Berkel}. Indeed, it is mentioned in Section 1 that $%
	\left\{ Y_{i}\right\} _{i=1}^{n}$ is positively associated. Moreover, $%
	\text{Cov}\left( Y_{j},Y_{j^{\prime }}\right) =\text{Cov}\left( X_{j},X_{j^{\prime
	}}\right) $ uniformly in $%
	%TCIMACRO{\U{2115} }%
	%BeginExpansion
	\mathbb{N}
	%EndExpansion
	$, since $\left\{ X_{i}\right\} _{i=1}^{n}$ and $\left\{ e_{i}\right\}
	_{i=1}^{n}$ are independent and the $e_{i}$'s are independent among
	themselves. So, in order to apply Lemma \ref{Berkel}, we shall calculate the
	first partial derivatives of transformations $\Phi _{j}\left( .\right) $
	which are $M_{h_{n}}\left( \frac{.}{h_{n}}\right) $ in our case and the sets 
	$A$ and $B$ present $1$ and $i$ respectively.
	
	Using the fact that  
	\[
	\frac{\partial}{\partial x} M_{h_{n}} \left( \frac{x}{h_{n}} \right) = \frac{1}{h_{n}} W_{h_{n}} \left( \frac{x}{h_{n}} \right),
	\]  
	and applying Lemma \ref{prop2}, we conclude that  
	\begin{equation}  
		\left\| \partial M_{h_{n}} \left( \frac{.}{h_{n}} \right) \right\|_{\infty} \leq \frac{C}{h_{n}^{\beta +1}}.  
		\label{H}  
	\end{equation} 
	
	Then, from the fact that $\text{Cov}\left( Y_{1},Y_{i}\right) =\text{Cov}\left(
	X_{1},X_{i}\right) $, Lemma \ref{Berkel} gives 
	\begin{eqnarray*}
		\left\vert \text{Cov}\left( \tilde{M}_{h_{n},1}\left( x\right) ,\tilde{M}%
		_{h_{n},i}\left( x\right) \right) \right\vert &=&\left\vert \text{Cov}\left(
		M_{h_{n}}\left( \frac{x-Y_{1}}{h_{n}}\right) ,M_{h_{n}}\left( \frac{x-Y_{i}}{%
			h_{n}}\right) \right) \right\vert \\
		&\leq &\left( \frac{C}{h_{n}^{\beta +1}}\right) ^{2}\text{Cov}\left(
		X_{1},X_{i}\right) .
	\end{eqnarray*}
	
	Consequently 
	\[
	h_{n}^{2\beta }\sum\limits_{i=\rho _{n}+1}^{n}\left\vert \text{Cov}\left( \tilde{M}%
	_{h_{n},1}\left( x\right) ,\tilde{M}_{h_{n},i}\left( x\right) \right)
	\right\vert \leq \frac{C}{h_{n}^{2}}\sum\limits_{i=\rho _{n}+1}^{n}\text{Cov}\left(
	X_{1},X_{i}\right) . 
	\]
	
	Using the fact that $\rho _{n}\leq i$, we can see $1\leq \left( \frac{i}{
		\rho _{n}}\right) ^{\mu }$. Thus 
	\[
	h_{n}^{2\beta }\sum\limits_{i=\rho _{n}+1}^{n}\left\vert \text{Cov}\left( \tilde{M}%
	_{h_{n},1}\left( x\right) ,\tilde{M}_{h_{n},i}\left( x\right) \right)
	\right\vert \leq \frac{Ch_{n}}{\rho _{n}^{\mu }h_{n}^{2}}\sum\limits_{i=\rho
		_{n}+1}^{+\infty }i^{\mu }\text{Cov}\left( X_{1},X_{i}\right) . 
	\]
	
	Now, we choose \( \rho _{n} = h_{n}^{-2/\mu} \). (We still have \( h_{n}^{2} \rho _{n} \to 0 \) as \( n \to +\infty \) since $\mu >1$.) Then, Condition (\textbf{H5})-2 simultaneously completes the proof of (\ref{item2}) and point (\ref{var1}) in Theorem \ref{variance}. Next, we proceed to establish item (\ref{cov1}). To this end, let us consider
	\begin{equation}
		\tilde{W}_{h_{n},i}\left( x\right) =\frac{1}{h_{n}}W_{h_{n}}\left( \frac{
			x-Y_{i}}{h_{n}}\right) -E\left[ \frac{1}{h_{n}}W_{h_{n}}\left( \frac{x-Y_{i} 
		}{h_{n}}\right) \right] .  \label{tild2}
	\end{equation}
	
	By stationarity, we can write 
	\begin{eqnarray}
		-\text{Cov}\left( f_{n}(x),1-F_{n}(x)\right) &=&\text{Cov}\left( f_{n}(x),F_{n}(x)\right) 
		\nonumber \\
		&=&\frac{1}{n}\text{Cov}\left( \tilde{M}_{h_{n},1}\left( x\right) ,\tilde{W}
		_{h_{n},1}\left( x\right) \right)  \nonumber \\
		&&+\frac{2}{n}\sum\limits_{i=2}^{n}\left( 1-\frac{i}{n}\right) \text{Cov}\left( 
		\tilde{M}_{h_{n},1}\left( x\right) ,\tilde{W}_{h_{n},i}\left( x\right)
		\right) .  \label{dec}
	\end{eqnarray}
	The goal now becomes
	
	\begin{equation}
		\lim_{n\rightarrow +\infty }h_{n}^{2\beta }\text{Cov}\left( \tilde{M}
		_{h_{n},1}\left( x\right) ,\tilde{W}_{h_{n},1}\left( x\right) \right)
		=G\left( x\right) D_{2},  \label{assur1}
	\end{equation}
	\begin{equation}
		h_{n}^{2\beta }\sum\limits_{i=2}^{n}\left\vert \text{Cov}\left( \tilde{M}
		_{h_{n},1}\left( x\right) ,\tilde{W}_{h_{n},i}\left( x\right) \right)
		\right\vert =o\left( 1\right) .  \label{assur2}
	\end{equation}
	Using the results in Proposition \ref{P1} together with Fubini's theorem, we
	have
	
	\begin{eqnarray*}
		\text{Cov}\left( \tilde{M}_{h_{n},1}\left( x\right) ,\tilde{W}_{h_{n},1}\left(
		x\right) \right) &=&\frac{1}{h_{n}}\int\limits_{-\infty }^{+\infty
		}W_{h_{n}}\left( \frac{x-u}{h_{n}}\right) \left[ \int\limits_{-\infty
		}^{x}W_{h_{n}}\left( \frac{s-u}{h_{n}}\right) ds\right] g(u)du+O(1) \\
		&=&\frac{1}{h_{n}}\int\limits_{-\infty }^{x}\left[ \int\limits_{-\infty
		}^{+\infty }W_{h_{n}}\left( \frac{x-u}{h_{n}}\right) W_{h_{n}}\left( \frac{
			s-u}{h_{n}}\right) g(u)du\right] ds+O(1).
	\end{eqnarray*}
	
	From Remark \ref{pdf}, we can see that: 
	\begin{eqnarray*}
		W_{h_{n}}\left( \frac{x-u}{h_{n}}\right) &=&W_{h_{n}}\left( \frac{s-u}{h_{n} 
		}+\frac{x-s}{h_{n}}\right) \\
		&=&W_{h_{n}}\left( \frac{s-u}{h_{n}}\right) +O(1).
	\end{eqnarray*}
	
	Therefore, since $g(.)$ is a density we have 
	\[
	\text{Cov}\left( \tilde{M}_{h_{n},1}\left( x\right) ,\tilde{W}_{h_{n},1}\left(
	x\right) \right) =\frac{1}{h_{n}}\int\limits_{-\infty }^{x}\left[
	\int\limits_{-\infty }^{+\infty }W_{h_{n}}^{2}\left( \frac{s-u}{h_{n}}%
	\right) g(u)du\right] ds+O(1). 
	\]
	
	Integration after a convergence using Lemma 1-b) of \cite{Masry2003}, lead
	to 
	\begin{eqnarray}
		\lim_{n\rightarrow +\infty }h_{n}^{2\beta }\text{Cov}\left( \tilde{M}
		_{h_{n},1}\left( x\right) ,\tilde{W}_{h_{n},1}\left( x\right) \right)
		&=&\lim_{n\rightarrow +\infty }\int\limits_{-\infty }^{x}\left[
		\int\limits_{-\infty }^{+\infty }\frac{1}{h_{n}}\kappa _{h_{n}}\left( \frac{
			s-u}{h_{n}}\right) g(u)du\right] ds+O(h_{n}^{2\beta })  \nonumber
		\label{cov} \\
		&=&G(x)D_{2}.  \nonumber
	\end{eqnarray}
	This ends the proof of (\ref{assur1}).
	
	One can reach point (\ref{assur2}) following similar arguments which used
	for (\ref{item2}). Then, we omit the details.
	
\end{proof}

\begin{proof}[Proof of Corollary \protect\ref{corellary1}]
	From Theorem \ref{variance}, we have  
	\begin{equation*}  
		nh_{n}^{2\beta } \text{Cov}(f_{n}(x),1-F_{n}(x)) = -D_{2}G(x)(1+o(1)),  
	\end{equation*}  
	
	\begin{equation*}  
		nh_{n}^{2\beta } \text{Var}(1-F_{n}(x)) = D_{1}G(x)\left[1-G(x)\right](1+o(1)).  
	\end{equation*}  
	
	Furthermore, Theorem 1 in \cite{Masry2003} gives  
	\begin{equation*}  
		nh_{n}^{2\beta } \text{Var}(f_{n}(x)) = D_{2}g(x)(1+o(1)).  
	\end{equation*}  
	
	Next, we consider the mapping  
	\begin{equation*}  
		\tau (x,y) = \frac{x}{y}, \quad \forall (x,y) \in \mathbb{R} \times \mathbb{R}^{\ast}.  
	\end{equation*}  
	
	Define  
	\begin{equation*}  
		\theta(x) := \left(f(x), 1-F(x)\right), \quad \theta_n(x) := \left(f_n(x), 1-F_n(x)\right).  
	\end{equation*}  
	
	Thus, we observe that \(\tau(\theta(x))\) and \(\tau(\theta_n(x))\) correspond to \(\lambda(x)\) and \(\lambda_n(x)\), respectively.  
	
	Now, we compute the variance of \(\theta_n(x)\):  
	\begin{equation}  
		\text{Var}(\theta_n(x)) := \Sigma_{\theta_n(x)} = \frac{1}{nh_n^{2\beta}}  
		\begin{pmatrix}  
			D_{2}g(x) & -D_{2}G(x) \\  
			-D_{2}G(x) & D_{1}G(x)[1-G(x)]  
		\end{pmatrix} (1+o(1)).  
		\label{var}  
	\end{equation}  
	
	By applying the \(\delta\)-method theorem due to Doob \cite{Doob1935}, we obtain  
	\begin{equation*}  
		\text{Var}(\lambda_n(x)) = \text{Var}(\tau(\theta_n(x))) \to \bigtriangledown \tau (\theta(x))^T \Sigma_{\theta(x)} \bigtriangledown \tau (\theta(x)),  
	\end{equation*}  
	where  
	\begin{equation*}  
		\Sigma_{\theta}(x) = \lim_{n\to \infty} \Sigma_{\theta_n}(x)  
	\end{equation*}  
	and the gradient vector \(\bigtriangledown \tau (\theta(x))\) is given by  
	\begin{equation}  
		\bigtriangledown \tau (\theta(x)) :=  
		\begin{pmatrix}  
			\frac{\partial \lambda(x)}{\partial f(x)} \\  
			\frac{\partial \lambda(x)}{\partial (1-F(x))}  
		\end{pmatrix}  
		=  
		\begin{pmatrix}  
			\frac{1}{1-F(x)} \\  
			\frac{f(x)}{(1-F(x))^2}  
		\end{pmatrix}.  
	\end{equation}  
	
	Thus, we obtain the asymptotic variance:  
	\begin{equation}  
		\lim_{n\to \infty} nh_n^{2\beta} \text{Var}(\lambda_n(x)) =  
		\frac{D_{2}g(x)}{(1-F(x))^2}  
		- 2 \frac{D_{2}f(x) G(x)}{(1-F(x))^3}  
		+ \frac{D_{1}[f(x)]^2 G(x)[1-G(x)]}{(1-F(x))^4}.  
		\label{asvar}  
	\end{equation}  
	
\end{proof}

\begin{proof}[Proof of Proposition \protect\ref{bais1}]
	Notice that: 
	\begin{equation}
		\frac{1}{1-F_{n}(x)} = \frac{1}{1-E[F_{n}(x)]} - \frac{E[F_{n}(x)] - F_{n}(x)}{[E(1-F_{n}(x))]^{2}} + \frac{(E[F_{n}(x)] - F_{n}(x))^{2}}{[1-F_{n}(x)] [E(1-F_{n}(x))]^{2}}.  
		\label{decomp}
	\end{equation}

	Taking expectation after multiplying by $f_{n}(x)$, we obtain  
	\begin{equation}
		E(\lambda _{n}(x)) = \tilde{E}(\lambda _{n}(x)) + \frac{s_{n}(x)+\text{Cov} ( f_n(x), F_{n}(x) )}{[E(1-F_{n}(x))]^{2}},
	\end{equation}
	where  
	\begin{equation*}
		s_{n}(x) =  E\Bigg[ \lambda_{n}(x) \Big( F_{n}(x) - E[F_{n}(x)] \Big)\Big( F_{n}(x) - E[F_{n}(x)] \Big) \Bigg].
	\end{equation*}
	
	From Theorem \ref{variance}, we obtain: $\text{Cov} ( f_n(x), F_{n}(x) )=O\left( \frac{1}{nh_{n}^{2\beta}} \right).$ 
	
	As for \( s_n(x) \), we can observe that:	
	\begin{align*}
		s_{n}(x) &= \text{Cov} \Bigg( \lambda_{n}(x) \Big( F_{n}(x) - E[F_{n}(x)] \Big), F_{n}(x)-E[F_{n}(x) \Bigg) \\
		&= \text{Cov} \Big( \lambda_{n}(x) F_{n}(x), F_{n}(x) \Big) - E[F_{n}(x)] \text{Cov} \Big( \lambda_{n}(x), F_{n}(x) \Big) \\
		&= -\text{Cov} \Bigg( \lambda_{n}(x) \Big(1 - F_{n}(x)\Big)-\lambda_{n}(x), F_{n}(x) \Bigg) - E[F_{n}(x)] \text{Cov} \Big( \lambda_{n}(x), F_{n}(x) \Big) \\
		&= -\text{Cov} \Big( f_n(x), F_{n}(x) \Big) + \text{Cov} \Big( \lambda_{n}(x), F_{n}(x) \Big) \Big(1 - E[F_{n}(x)]\Big)
	\end{align*}
	The first term is done. For the second term, applying Theorem \ref{variance}, Corollary \ref{corellary1}, and the Cauchy-Schwarz inequality, we get:
	\begin{align*}
		\text{Cov} \Big( \lambda_{n}(x), F_{n}(x) \Big) &\leq \sqrt{\text{Var}(\lambda_{n}(x))} \cdot \sqrt{\text{Var}(F_{n}(x))} \\
		&= O\left( \frac{1}{n h_{n}^{2\beta}} \right).
	\end{align*}
The desired result follows from Proposition \ref{P1}, which states that $E[1 - F_{n}(x)] \to 1 - F(x) > 0 \quad \text{as} \quad n \to \infty.$

\end{proof}

\begin{proof}[Proof of Proposition \ref{bais2}]
	Using simple algebraic manipulations, we derive:
	\[
	\tilde{E}\left( \lambda_{n}(x) \right) - \lambda(x) = \frac{E\left( f_{n}(x) \right) - \lambda(x) \left[ 1 - E\left( F_{n}(x) \right) \right]}{1 - E\left( F_{n}(x) \right)}.
	\]
	This can be rewritten as:
	\[
	\tilde{E}\left( \lambda_{n}(x) \right) - \lambda(x) = \frac{E\left[ f_{n}(x) - f(x) \right] + \lambda(x) E\left[ F_{n}(x) - F(x) \right]}{1 - E\left( F_{n}(x) \right)}.
	\]
	
	For the denominator, applying the convolution theorem yields:
	\[
	1 - E\left( F_{n}(x) \right) = 1 - \int_{-\infty}^{+\infty} K(t) F(x - t h_{n}) \, dt.
	\]
	Under condition (\textbf{H2}) and using a first-order Taylor expansion, we obtain:
	\[
	1 - E\left( F_{n}(x) \right) = 1 - F(x) + o(1).
	\]
	
	For the numerator, a second-order Taylor expansion gives:
	\[
	\tilde{E}\left( \lambda_{n}(x) \right) - \lambda(x) = \frac{\frac{h_{n}^{2}}{2} \left[ f^{\prime \prime}(x) + \lambda(x) F^{\prime \prime}(x) \right] \int_{-\infty}^{+\infty} t^{2} k(t) \, dt}{1 - F(x)} + o(1).
	\]
	Simplifying further, we have:
	\[
	\tilde{E}\left( \lambda_{n}(x) \right) - \lambda(x) = \frac{h_{n}^{2}}{2} \left[ \lambda^{\prime \prime}(x) - \frac{2 \lambda^{\prime}(x) F^{\prime}(x)}{1 - F(x)} \right] \int_{-\infty}^{+\infty} t^{2} k(t) \, dt + o(1).
	\]
	This completes the proof.
\end{proof}

\begin{proof}[Proof of Theorem \protect\ref{an}]
	The goal is to establish the following convergence in distribution:
	\[
	\left( nh_{n}^{2\beta} \right)^{1/2} \left[ \begin{pmatrix} f_{n}(x) \\ 1 - F_{n}(x) \end{pmatrix} - \begin{pmatrix} E\left[ f_{n}(x) \right] \\ E\left[ 1 - F_{n}(x) \right] \end{pmatrix} \right] \overset{\mathcal{D}}{\rightarrow} \mathcal{N}_{2} \left( \mathbf{0}, \Sigma(x) \right),
	\]
	where \(\mathcal{N}_{2} \left( \mathbf{0}, \Sigma(x) \right)\) denotes a bivariate normal distribution with mean vector \(\mathbf{0} = (0, 0)^T\) and asymptotic covariance matrix:
	\[
	\Sigma(x) = \begin{pmatrix}
		D_{2}g(x) & -D_{2}G(x) \\
		-D_{2}G(x) & D_{1}G(x) \left[ 1 - G(x) \right]
	\end{pmatrix}.
	\]
	
	Let \(\alpha = (\alpha_1, \alpha_2)^T \in \mathbb{R}^2\) such that \(\alpha_1^2 + \alpha_2^2 \neq 0\). The goal reduces to showing:
	\[
	\sqrt{nh_{n}^{2\beta}} \left[ \alpha_1 \tilde{f}_n(x) - \alpha_2 \tilde{F}_n(x) \right] \overset{\mathcal{D}}{\rightarrow} \mathcal{N} \left( 0, \sigma^2(x) \right),
	\]
	where \(\tilde{f}_n(x) = f_n(x) - E\left[ f_n(x) \right]\), \(\tilde{F}_n(x) = F_n(x) - E\left[ F_n(x) \right]\), and:
	\[
	\sigma^2(x) = \alpha^T \Sigma(x) \alpha. \label{sigma}
	\]
	
	To proceed, define:
	\[
	\Lambda_n(x) = \sqrt{nh_{n}^{2\beta}} \left( \alpha_1 \tilde{f}_n(x) - \alpha_2 \tilde{F}_n(x) \right).
	\]
	
	From the definitions of \(\tilde{M}_{h_n,i}(x)\) and \(\tilde{W}_{h_n,i}(x)\) in \(\eqref{tild1}\) and \(\eqref{tild2}\), respectively, we have:
	\[
	\tilde{F}_n(x) = \frac{1}{nh_n} \sum_{i=1}^n h_n \tilde{M}_{h_n,i}(x),
	\]
	\[
	\tilde{f}_n(x) = \frac{1}{nh_n} \sum_{i=1}^n \tilde{W}_{h_n,i}(x).
	\]
	
	For notational convenience, let:
	\[
	T_{h_n,i}(x) = \alpha_1 \tilde{W}_{h_n,i}(x) - \alpha_2 h_n \tilde{M}_{h_n,i}(x). \label{T}
	\]
	It follows that:
	\[
	\Lambda_n(x) = \frac{1}{\sqrt{n}} \sum_{i=1}^n h_n^{\beta - 1} T_{h_n,i}(x) = \frac{1}{\sqrt{n}} \sum_{i=1}^n U_i.
	\]
	
	Thus, we aim to show:
	\[
	\Lambda_n(x) \overset{\mathcal{D}}{\rightarrow} \mathcal{N} \left( 0, \alpha^T \Sigma(x) \alpha \right).
	\]
	
	First, from Theorem \ref{variance}, we observe:
	\[
	\lim_{n \to \infty} \text{Var} \left( \Lambda_n(x) \right) = \alpha_1^2 \lim_{n \to \infty} \text{Var} \left( \sqrt{nh_n^{2\beta}} f_n(x) \right) 
	+ \alpha_2^2 \lim_{n \to \infty} \text{Var} \left( \sqrt{nh_n^{2\beta}} (1 - F_n(x)) \right) 
	\]
	\[
	+ 2\alpha_1 \alpha_2 \lim_{n \to \infty} \text{Cov} \left( \sqrt{nh_n^{2\beta}} f_n(x), \sqrt{nh_n^{2\beta}} (1 - F_n(x)) \right) = \alpha^T \Sigma(x) \alpha.
	\]

	To establish asymptotic normality for dependent random variables, we employ the big-block and small-block technique. Using the notation in Condition (\textbf{H6}), the set \(\{1, 2, \dots, n\}\) is partitioned into \(2\tau_n + 1\) subsets, consisting of large blocks of size \(p_n = p\) and small blocks of size \(q_n = q\), respectively, where \(\tau_n = \left\lfloor \frac{n}{p_n + q_n} \right\rfloor\). Specifically, define:
	\[
	T_s = \left\{ i : i = (s-1)(p_n + q_n) + 1, \dots, (s-1)(p_n + q_n) + p_n \right\}, \label{large}
	\]
	\[
	I_s = \left\{ j : j = (s-1)(p_n + q_n) + p_n + 1, \dots, s(p_n + q_n) \right\}. \label{small}
	\]
	
	For \(1 \leq s \leq \tau_n\), define the random variables:
	\[
	\eta_s = \sum_{i=(s-1)(p_n + q_n) + 1}^{(s-1)(p_n + q_n) + p_n} U_i \quad \text{and} \quad \zeta_s = \sum_{j=(s-1)(p_n + q_n) + p_n + 1}^{s(p_n + q_n)} U_i. \label{variables}
	\]
	
	The remaining block is:
	\[
	\vartheta_{\tau_n} = \sum_{i=\tau_n(p_n + q_n) + 1}^n U_i.
	\]
	
	Thus, we can write:
	\[
	\frac{1}{\sqrt{n}} \sum_{i=1}^n U_i = \frac{1}{\sqrt{n}} \left[ \sum_{s=1}^{\tau_n} \eta_s + \sum_{s=1}^{\tau_n} \zeta_s + \vartheta_{\tau_n} \right] = \frac{1}{\sqrt{n}} \left[ I_{n,1} + I_{n,2} + I_{n,3} \right].
	\]
	
	The goal now is to show:
	\[
	\frac{1}{n} E\left[ I_{n,2}^2 \right] \to 0, \quad \frac{1}{n} E\left[ I_{n,3}^2 \right] \to 0, \label{negligible}
	\]
	\[
	\frac{1}{\sqrt{n}} I_{n,1} \overset{\mathcal{L}}{\rightarrow} \mathcal{N} \left( 0, \sigma^2(x) \right). \label{normality}
	\]
	
	The limits in \(\eqref{negligible}\) indicate that \(I_{n,2}\) and \(I_{n,3}\) are asymptotically negligible, while \(\eqref{normality}\) establishes the asymptotic normality of the leading term \(I_{n,1}\). To prove \(\eqref{negligible}\), we rely on Lemma \ref{lemma1}.
	
	\begin{lemma}
		\label{lemma1} Under Conditions  (\textbf{H1}), (\textbf{H3})--(\textbf{H4}), (\textbf{H5})-2, (\textbf{H6}) , and for sufficiently large \( n \), we have:

		i) $\frac{\tau _{n}}{n}\text{Var}\left( \zeta _{1}\right) \rightarrow 0,$
		
		ii)$\frac{1}{n}\sum\limits_{1\leq i<j\leq \tau _{n}}\text{Cov}\left( \zeta
		_{i},\zeta _{j}\right) \rightarrow 0,$
		
		iii) $\text{Var}\left( \frac{1}{\sqrt{n}}I_{n,2}\right) \rightarrow 0,$
		
		iv) $\text{Var}\left( \frac{1}{\sqrt{n}}I_{n,3}\right) \rightarrow 0.$
	\end{lemma}
	
	\begin{proof}[Proof of Lemma \protect\ref{lemma1}]
		From hypothesis (\textbf{H6}), we can see that  
		\begin{equation}
			\frac{q_{n}}{p_{n}}\rightarrow 0,\frac{p_{n}}{n}\rightarrow 0,\frac{\tau _{n}%
			}{nh_{n}}\rightarrow 0,\frac{p_{n}}{\sqrt{nh_{n}}}\rightarrow 0.
			\label{conc}
		\end{equation}
		
		\textbf{Proof of i)}: The stationarity leads to
		
		\begin{eqnarray}  \label{decom}
			\text{Var}\left( \zeta _{1}\right) &=&q_{n}\text{Var}\left( U_{1}\right)
			+2\sum\limits_{1\leq i<j\leq q_{n}}\left\vert \text{Cov}\left( U_{i},U_{j}\right)
			\right\vert  \notag \\
			&:&=\mathcal{T}_{n,1}+\mathcal{T}_{n,2}.
		\end{eqnarray}
		
		We deal with $\mathcal{T}_{n,1}$ and $\mathcal{T}_{n,2}$ separately. 
		From the previous analysis, we can see that  
		\begin{eqnarray}  \label{relation1}
			\text{Var}\left( U_{1}\right) &=&h_{n}^{2\beta -2}\text{Var}\left( T_{h_{n},1}\right) 
			\notag \\
			&=&h_{n}^{2\beta -2}\left[ \alpha _{1}^{2}\text{Var}\left( \tilde{W}
			_{h_{n},1}\left( x\right) \right) +\alpha _{2}^{2}h_{n}^{2}\text{Var}\left( \tilde{%
				M }_{h_{n},1}\left( x\right) \right) \right.  \notag \\
			&&\left. -2\alpha _{1}\alpha _{2}h_{n}\text{Cov}\left( \tilde{W}
			_{h_{n},1}\left( x\right) ,\tilde{M}_{h_{n},1}\left( x\right) \right) \right]
			\notag \\
			&\rightarrow &\sigma ^{2}\left( x\right) .
		\end{eqnarray}
		
		Concerning $\mathcal{T}_{n,2}$, by stationarity we get  
		\begin{eqnarray}  \label{ralation2}
			\mathcal{T}_{n,2} &=&2\sum\limits_{1\leq i<j\leq q_{n}}\left\vert \text{Cov}\left(
			U_{i},U_{j}\right) \right\vert ,  \notag \\
			&=&2\sum\limits_{l=1}^{q_{n}}\left( q_{n}-l\right) \left\vert \text{Cov}\left(
			U_{1},U_{l}\right) \right\vert ,  \notag \\
			&\leq &2q_{n}\sum\limits_{l=2}^{q_{n}}h_{n}^{2\beta -2}\left\vert \text{Cov}\left(
			\left( \alpha _{1}\tilde{W}_{h_{n},1}\left( x\right) -\alpha _{2}h_{n} 
			\tilde{M}_{h_{n},1}\left( x\right) \right) ,\right. \right.  \notag \\
			&&\left. \left. \left( \alpha _{1}\tilde{W}_{h_{n},l}\left( x\right)
			-\alpha _{2}h_{n}\tilde{M}_{h_{n},l}\left( x\right) \right) \right) \right] ,
			\notag \\
			&=&o\left( q_{n}\right) .
		\end{eqnarray}
		
		Combining the results in (\ref{decom}), (\ref{relation1}), and (\ref%
		{ralation2}) we get  
		\begin{equation}  \label{pn}
			\text{Var}\left( \zeta _{1}\right) =q_{n}\sigma ^{2}\left( x\right) \left(
			1+o\left( 1\right) \right) .
		\end{equation}
		
		Thus 
		\begin{equation*}
			\frac{\tau _{n}}{n}\text{Var}\left( \zeta _{1}\right) =\frac{\tau _{n}}{n}
			q_{n}\sigma ^{2}\left( x\right) \left( 1+o\left( 1\right) \right) \sim \frac{
				q_{n}}{q_{n}+p_{n}}\sim \frac{q_{n}}{p_{n}}.
		\end{equation*}
		
		The proof follows from (\ref{conc}).
		
		\textbf{Proof of ii): }
		
		At this level, we consider $\gamma _{s}=s\left( q_{n}+p_{n}\right) +q_{n}$. 
		Hence, we have
		
		\begin{equation*}
			\frac{1}{n}\sum\limits_{1\leq i<j\leq \tau _{n}}\text{Cov}\left( \zeta _{i},\zeta
			_{j}\right) =\frac{1}{n}\sum\limits_{1\leq i<j\leq \tau
				_{n}}\sum\limits_{r_{1}=1}^{q_{n}}\sum\limits_{r_{2}=1}^{q_{n}}\text{Cov}\left\{
			U_{\gamma _{i}+r_{1}},U_{\gamma _{j}+r_{2}}\right\} .
		\end{equation*}
		
		It is may easy to see that $\left\vert \gamma _{i}-\gamma
		_{j}+r_{1}+r_{2}\right\vert \geq q_{n}$ since $j\neq i$. Hence 
		\begin{equation*}
			\frac{1}{n}\sum\limits_{1\leq i<j\leq q_{n}}\left\vert \text{Cov}\left( \zeta
			_{i},\zeta _{j}\right) \right\vert \leq \frac{2}{n}\sum\limits_{r_{1}=1}^{n-%
				\tau _{n}}\sum\limits_{r_{2}=r_{1}+q_{n}}^{n}\left\vert \text{Cov}\left\{
			U_{r_{1}},U_{r_{2}}\right\} \right\vert .
		\end{equation*}
		
		By stationarity, we can see that  
		\begin{eqnarray*}
			\frac{1}{n}\sum\limits_{1\leq i<j\leq q_{n}}\left\vert \text{Cov}\left( \zeta
			_{i},\zeta _{j}\right) \right\vert &\leq
			&2\sum\limits_{s=q_{n}}^{n}\text{Cov}\left\{ U_{1},U_{s}\right\} . \\
			&=&o\left( 1\right) .
		\end{eqnarray*}
		
		\textbf{Proof of iii):} Again by stationarity, we have 
		\begin{equation*}
			\text{Var}\left( \frac{1}{\sqrt{n}}I_{n,2}\right) =\frac{\tau _{n}}{n}\text{Var}\left(
			\zeta _{1}\right) +\frac{2}{n}\sum\limits_{1\leq i<j\leq \tau
				_{n}}\left\vert \text{Cov}\left( \zeta _{i},\zeta _{j}\right) \right\vert .
		\end{equation*}
		
		Then, the proof finishes using items i) and ii).
		
		\textbf{Proof of vi): }By analogous argument, we find  
		\begin{eqnarray*}
			\text{Var}\left( \frac{1}{\sqrt{n}}I_{n,3}\right) &=&\frac{1}{n}\text{Var}\left(
			\sum\limits_{i=\tau _{n}\left( p_{n}+q_{n}\right) +1}^{n}U_{i}\right) \\
			&=&\frac{n-\tau _{n}\left( p_{n}+q_{n}\right) }{n}\text{Var}\left( U_{1}\right) + 
			\frac{2}{n}\sum\limits_{\tau _{n}\left( p_{n}+q_{n}\right) +1\leq i<j\leq
				n}\left\vert \text{Cov}\left( U_{i},U_{j}\right) \right\vert . \\
			&=&\frac{n-\tau _{n}\left( p_{n}+q_{n}\right) }{n}\text{Var}\left( U_{1}\right) + 
			\frac{2}{n}\sum\limits_{1\leq i<j\leq n-\tau _{n}\left( p_{n}+q_{n}\right)
			}\left\vert \text{Cov}\left( U_{i},U_{j}\right) \right\vert . \\
			&:&=\mathcal{F}_{n,1}+\mathcal{F}_{n,2}.
		\end{eqnarray*}
		
		First, simple calculation together with (\ref{conc}) lead to 
		\begin{eqnarray*}
			\mathcal{F}_{n,1} &=&\frac{n-\tau _{n}\left( p_{n}+q_{n}\right) }{n}%
			\text{Var}\left( U_{1}\right) \\
			&\leq &\frac{p_{n}}{n}\sigma ^{2}\left( x\right) \\
			&=&o\left( 1\right) .
		\end{eqnarray*}
		
		For $F_{n,2}$, we can write  
		\begin{eqnarray*}
			\mathcal{F}_{n,2} &=&\frac{2}{n}\sum\limits_{1\leq i<j\leq n-\tau _{n}\left(
				p_{n}+q_{n}\right) }\left\vert \text{Cov}\left( U_{i},U_{j}\right) \right\vert \\
			&\leq &\frac{2}{n}\sum\limits_{1\leq i<j\leq p_{n}}\left\vert \text{Cov}\left(
			U_{i},U_{j}\right) \right\vert \\
			&\leq &\frac{2p_{n}}{n}\sum\limits_{s=1}^{p_{n}-1}\left\vert \text{Cov}\left(
			U_{1},U_{1+s}\right) \right\vert \\
			&\leq &\frac{2p_{n}}{n}\sum\limits_{s=1}^{n}\left\vert \text{Cov}\left(
			U_{1},U_{1+s}\right) \right\vert
		\end{eqnarray*}
		
		By a similar argument we find, using (\ref{conc}):  
		\begin{equation*}
			\mathcal{F}_{n,2}=o\left( 1\right) .
		\end{equation*}
		
		This ends simultaneously the proofs of Lemma \ref{lemma1} and that of (\ref%
		{negligible}).
	\end{proof}
	
	We now proceed to establish the proof of (\ref{normality}). For this, we
	give the following assertions%
	\begin{equation}
		\frac{1}{n}\sum\limits_{s=1}^{\tau _{n}}E\left[ \eta _{s}^{2}\right]
		\rightarrow \sigma ^{2}\left( x\right) ,  \label{Lindeberg1}
	\end{equation}%
	\begin{equation}
		\left\vert E\left[ e^{ \frac{it}{\sqrt{n}}I_{n,1}} \right]
		-\prod\limits_{s=1}^{\tau _{n}}E\left[ e^ {\frac{it}{\sqrt{n}}\eta
			_{s}} \right] \right\vert \rightarrow 0,  \label{independent}
	\end{equation}
	
	\begin{equation}
		\frac{1}{n}\sum\limits_{s=1}^{\tau _{n}}E\left[ \eta _{s}^{2}1\left\{
		\left\vert \eta _{s}^{2}\right\vert >\varepsilon \sqrt{n}\sigma \left(
		x\right) \right\} \right] \rightarrow 0.  \label{Lindeberg2}
	\end{equation}
	
	Assertion (\ref{Lindeberg1}) establishes the explicit expression for the
	limit variance of the leading term $I_{n,1}$. Point (\ref{independent})
	demonstrates the asymptotic independence of the r.v.'s $\left\{ \eta
	_{s}\right\} $ composing the sum $I_{n,1}$ using the characteristic function
	criterion. Assurtion (\ref{Lindeberg2}) presents the
	Lindeberg-Feller condition of the asymptotic normality under independence. 
	
	First, we proceed to prove assertion (\ref{Lindeberg1}). Lemma \ref{lemma2} hereinafter
	is in order:
	
	\begin{lemma}
		\label{lemma2} Under Conditions (\textbf{H1}), (\textbf{H3})--(\textbf{H4}), (\textbf{H5})-2, and (\textbf{H6}), for sufficiently large \( n \), we have:

		i) $\frac{\tau _{n}}{n}\text{Var}\left( \eta _{1}\right) \rightarrow \sigma 
		^{2}\left( x\right) ,$
		
		ii)$\frac{1}{n}\sum\limits_{1\leq i<j\leq \tau _{n}}\text{Cov}\left( \eta _{i},\eta
		_{j}\right) \rightarrow 0,$
		
		iii) $\text{Var}\left( \frac{1}{\sqrt{n}}I_{1,n}\right) \rightarrow \sigma 
		^{2}\left( x\right) .$
	\end{lemma}
	
	\begin{proof}[Proof of Lemma \protect\ref{lemma2}]
		\textbf{Proof of i)}: By means of (\ref{variables}) and stationarity, we
		have 
		\begin{eqnarray*}
			\frac{\tau _{n}}{n}\text{Var}\left( \eta _{1}\right)  &=&\frac{\tau _{n}}{n}%
			\text{Var}\left( \sum\limits_{i=1}^{p_{n}}U_{i}\right)  \\
			&=&\frac{\tau _{n}p_{n}}{n}\text{Var}\left( U_{1}\right) +\frac{2\tau _{n}}{n}%
			\sum\limits_{1\leq i<j\leq p_{n}}\left\vert \text{Cov}\left( U_{i},U_{j}\right)
			\right\vert  \\
			&:&=\mathcal{T}_{n,1}^{\prime }+\mathcal{T}_{n,2}^{\prime }.
		\end{eqnarray*}
		
		The calculation of $\mathcal{T}_{n,1}^{\prime }$ and $\mathcal{T}%
		_{n,2}^{\prime }$ with $p_{n}$ terms is identical to that of $\mathcal{T}%
		_{n,1}$ and $\mathcal{T}_{n,2}$ with $q_{n}$ terms in (\ref{decom}) respectively.
		Thus,  $p_{n}$ replacing $q_{n}$ in (\ref{pn}) lead to 
		\begin{equation*}
			\frac{\tau _{n}}{n}\text{Var}\left( \eta _{1}\right) =\frac{\tau _{n}p_{n}}{n}%
			\sigma ^{2}\left( x\right) \left( 1+o\left( 1\right) \right) .
		\end{equation*}%
		The proof follows immediately from the fact that $\frac{\tau _{n}p_{n}}{n}%
		\rightarrow 1$ as $n\rightarrow +\infty $.
		
		\textbf{Proof of ii): } Following the same steps of proving item ii) in Lemma 1, we get 
		
		\begin{eqnarray*}
			\frac{1}{n}\sum\limits_{1\leq i<j\leq \tau _{n}}\left\vert \text{Cov}\left( \eta
			_{i},\eta _{j}\right) \right\vert  &\leq
			&2\sum\limits_{s=p_{n}}^{n}\text{Cov}\left\{ U_{1},U_{s}\right\} . \\
			&=&o\left( 1\right) .
		\end{eqnarray*}
		
		\textbf{Proof of iii)}: By stationarity, we can write%
		\begin{equation*}
			\text{Var}\left( \frac{1}{\sqrt{n}}I_{n,1}\right) =\frac{\tau _{n}}{n}\text{Var}\left(
			\eta _{1}\right) +\frac{2}{n}\sum\limits_{1\leq i<j\leq \tau _{n}}\left\vert
			\text{Cov}\left( \eta _{i},\eta _{j}\right) \right\vert 
		\end{equation*}
		
		The result follows using items i) and ii).
	\end{proof}
	
	Next, we move to prove (\ref{independent}) using the fact that $\left\{
	Y_{i}\right\} _{i\geq 1}$ is an associated random process. The strategy here
	is to use Lemma  \ref{Berkel} to bound the left-hand side of (\ref{independent}) in terms of
	the covariance sequences of the process $\left\{ X_{i}\right\} _{i\geq 1}$.
	Let us define 
	\[
	I_{\tau _{n}}\left( t\right) :=\left\vert E\left[ e^{\frac{it}{\sqrt{n}}%
		I_{n,1}}\right] -\prod\limits_{s=1}^{\tau _{n}}E\left[ e^{\frac{it}{\sqrt{n}}%
		\eta _{s}}\right] \right\vert .
	\]
	
	We proceed as follows
	
	\begin{eqnarray*}
		I_{\tau _{n}}\left( t\right)  &\leq &\left\vert E\left[ e^{it\sqrt{n}%
			\sum\limits_{s=1}^{\tau }\eta _{s}}\right] -E\left[ e^{it\sqrt{n}%
			\sum\limits_{s=1}^{\tau _{n}-1}\eta _{s}}\right] E\left[ e^{it\sqrt{n}\eta
			_{\tau _{n}}}\right] \right\vert  \\
		&&+\left\vert E\left[ e^{it\sqrt{n}\sum\limits_{s=1}^{\tau _{n}-1}\eta _{s}}%
		\right] E\left[ e^{it\sqrt{n}\eta _{\tau _{n}}}\right] -\prod\limits_{s=1}^{%
			\tau _{n}}E\left[ e^{it\sqrt{n}\eta _{s}}\right] \right\vert 
	\end{eqnarray*}%
	Thus, it is clearly to see that 
	\[
	I_{\tau _{n}}\left( t\right) \leq \left\vert \text{Cov}\left( e^{it\sqrt{n}%
		\sum\limits_{s=1}^{\tau _{n}-1}\eta _{s}},e^{it\sqrt{n}\eta _{\tau
			_{n}}}\right) \right\vert +I_{\tau _{n}-1}\left( t\right) .
	\]%
	We repeat this recursive process for the term $I_{\tau _{n}-1}\left(
	t\right) $ and so on, we get 
	\begin{equation}
		I_{\tau _{n}}\left( t\right) \leq \sum\limits_{l=1}^{\tau _{n}-l}\left\vert
		\text{Cov}\left( e^{it\sqrt{n}\sum\limits_{s=1}^{\tau _{n}-l}\eta _{s}},e^{it\sqrt{n%
			}\eta _{\tau _{n}-l+1}}\right) \right\vert .  \label{sum1}
	\end{equation}
	
	Using the large block sets $L_{s}$ in (\ref{large}), the right-hand side of (%
	\ref{sum1}) becomes 
	\[
	\sum\limits_{l=1}^{\tau _{n}}\left\vert \text{Cov}\left( e^{\left( \frac{it}{\sqrt{n%
		}}\sum\limits_{i\in L_{0}\cup L_{1}...L_{\tau _{n}-1-l}}\Psi \left(
		Y_{i}\right) \right) },e^{\left( \frac{it}{\sqrt{n}}\sum\limits_{i\in
			L_{\tau _{n}-l}}\Psi \left( Y_{i}\right) \right) }\right) \right\vert ,
	\]

	where $\Psi \left( Y_{i}\right) =\left( \alpha _{1}W_{h_{n}}\left( \frac{%
		x-Y_{i}}{h_{n}}\right) -\alpha _{2}h_{n}M_{h_{n}}\left( \frac{x-Y_{i}}{h_{n}}%
	\right) \right) $. From Lemma \ref{prop2}, it seen that%
	\[
	\Psi _{1}\left( Y_{i}\right) :=e^{\frac{it}{\sqrt{n}}\Psi \left(
		Y_{i}\right) }
	\]
	
	has a bounded derivative. Furthermore%
	\begin{eqnarray*}
		\left\Vert \Psi _{1}^{\prime }\left( x\right) \right\Vert _{\infty } &\leq & 
		\frac{\left\vert t\right\vert }{\sqrt{n}}\left[ \alpha _{1}\frac{1}{h_{n}}
		\left\Vert W _{h_{n}}\right\Vert _{\infty }+\alpha _{2}\left\Vert
		M_{h_{n}}\right\Vert _{\infty }\right] \\
		&\leq &\frac{C\left\vert t\right\vert }{h_{n}^{\beta +1}\sqrt{n}}.
	\end{eqnarray*}
	
	Then, by Lemma \ref{Berkel} we get 
	\[
	I_{\tau _{n}}\leq \frac{C\left\vert t\right\vert ^{2}}{nh_{n}^{2\left( \beta
			+1\right) }}\sum\limits_{l=1}^{\tau _{n}}\sum\limits_{i\in L_{0}\cup
		L_{1}\cup ...L_{\tau _{n}-1-l}}\sum\limits_{j\in L_{\tau _{n}-l}}\left\vert
	\text{Cov}\left( Y_{i},Y_{j}\right) \right\vert
	\]%
	The stationarity beside the fact that $\text{Cov}\left( Y_{i},Y_{j}\right)
	=\text{Cov}\left( X_{i},X_{j}\right) $, we get 
	\begin{eqnarray}  \label{fin}
		I_{\tau _{n}}\left( t\right) &\leq &\frac{Ct^{2}}{nh_{n}^{2\left( \beta
				+1\right) }}\left[ \left( \tau _{n}-1\right) \sum\limits_{j\in
			L_{0}}\sum\limits_{i\in L_{1}}\left\vert \text{Cov}\left( X_{i},X_{j}\right)
		\right\vert +\left( \tau _{n}-2\right) \sum\limits_{j\in
			L_{0}}\sum\limits_{i\in L_{2}}\left\vert \text{Cov}\left( X_{i},X_{j}\right)
		\right\vert ..\right.  \nonumber \\
		&&+\left. \sum\limits_{j\in L_{0}}\sum\limits_{i\in L_{\tau _{n}}}\left\vert
		\text{Cov}\left( X_{i},X_{j}\right) \right\vert \right]
	\end{eqnarray}%
	Furthermore, again by stationarity, we have
	
	\begin{eqnarray*}
		\sum\limits_{i\in L_{0}}\sum\limits_{j\in L_{1}}\left\vert \text{Cov}\left(
		X_{i},X_{j}\right) \right\vert
		&=&\sum\limits_{i=1}^{p_{n}}\sum\limits_{j=\left( p_{n}+q_{n}\right)
		}^{\left( p_{n}+q_{n}\right) +p_{n}-1}\left\vert \text{Cov}\left(
		X_{i},X_{j}\right) \right\vert \\
		&=&p_{n}\left\vert \text{Cov}\left( X_{1},X_{p_{n}+q_{n}}\right) \right\vert
		+\left( p_{n}-1\right) \left\vert \text{Cov}\left( X_{1},X_{p_{n}+q_{n}+1}\right)
		\right\vert ... \\
		&&+\left\vert \text{Cov}\left( X_{1},X_{\left( p_{n}+q_{n}\right) +p_{n}-1}\right)
		\right\vert \\
		&=&\sum\limits_{r=1}^{p_{n}}\left( p_{n}-r\right) \left\vert \text{Cov}\left(
		X_{1},X_{p_{n}+q_{n}+r}\right) \right\vert .
	\end{eqnarray*}
	Analogously
	
	\begin{equation*}
		\sum\limits_{j\in L_{0}}\sum\limits_{i\in L_{2}}\left\vert \text{Cov}\left(
		X_{i},X_{j}\right) \right\vert =\sum\limits_{r=1}^{p_{n}}\left(
		p_{n}-r\right) \left\vert \text{Cov}\left( X_{1},X_{2\left( p_{n}+q\right)
			_{n}+r}\right) \right\vert .
	\end{equation*}
	
	Finally
	
	\begin{equation*}
		\sum\limits_{j\in L_{0}}\sum\limits_{i\in T_{\tau _{_{n}}-1}}\left\vert
		\text{Cov}\left( X_{i},X_{j}\right) \right\vert =\sum\limits_{r=1}^{p_{n}}\left(
		p_{n}-r\right) \left\vert \text{Cov}\left( X_{1},X_{\left( \tau _{n}-1\right)
			\left( p_{n}+q\right) _{n}+r}\right) \right\vert .
	\end{equation*}%
	Hence, (\ref{fin}) becomes
	
	\begin{equation*}
		I_{\tau _{n}}\left( t\right) \leq \frac{Ct^{2}}{nh_{n}^{2\left( \beta
				+1\right) }}\sum\limits_{l=0}^{\tau _{n}}\sum\limits_{r=1}^{p_{n}}\left(
		\tau _{n}-l\right) \left( p_{n}-r\right) \left\vert \text{Cov}\left(
		X_{1},X_{l\left( p_{n}+q_{n}\right) +r}\right) \right\vert .
	\end{equation*}%
	We set the following change of index: $i=l\left( p_{n}+q_{n}\right) +r$.
	Hence, we have 
	\begin{equation*}
		I_{\tau _{n}}\left( t\right) \leq \frac{Ct^{2}}{nh_{n}^{2\left( \beta
				+1\right) }}\sum\limits_{r=1}^{p_{n}}\left( p_{n}-r\right)
		\sum\limits_{i=\left( p_{n}+q_{n}\right) +r}^{\tau _{n}\left(
			p_{n}+q_{n}\right) +r}\left( \tau _{n}-\frac{i-r}{\left( p_{n}+q_{n}\right) }%
		\right) \left\vert \text{Cov}\left( X_{1},X_{i}\right) \right\vert .
	\end{equation*}
	
	Next 
	\begin{eqnarray*}
		I_{\tau _{n}}\left( t\right)  &\leq &\frac{Ct^{2}\tau _{n}}{nh_{n}^{2\left(
				\beta +1\right) }}\sum\limits_{r=1}^{p_{n}}\left( p_{n}-r\right)
		\sum\limits_{i=\left( p_{n}+q_{n}\right) +r}^{\tau _{n}\left(
			p_{n}+q_{n}\right) +r}\left\vert \text{Cov}\left( X_{1},X_{i}\right) \right\vert  \\
		&\leq &\frac{Ct^{2}\tau _{n}p_{n}}{nh_{n}^{2\left( \beta +1\right) }}%
		\sum\limits_{i=p_{n}}^{+\infty }\left\vert \text{Cov}\left( X_{1},X_{i}\right)
		\right\vert .
	\end{eqnarray*}
	
	Finally, from the fact that $\frac{\tau _{n}p_{n}}{n}\rightarrow 1$ as $%
	n\rightarrow +\infty $ and Condition (\textbf{H6}), we conclude that 
	\begin{equation*}
		I_{\tau _{n}}\left( t\right) \leq \frac{Ct^{2}}{h_{n}^{2\left( \beta
				+1\right) }}\sum\limits_{i=p_{n}}^{+\infty }\left\vert \text{Cov}\left(
		X_{1},X_{i}\right) \right\vert \rightarrow 0.
	\end{equation*}
	
	The proof of (\ref{independent}) is finished.
	
	Now, we establish (\ref{Lindeberg2}). From Lemma \ref{prop2}, we can see
	that $\left\vert W _{h_{n}}\left( x\right) \right\vert \leq \frac{C}{%
		h_{n}^{\beta }}$ and $\left\vert M_{h_{n}}\left( x\right) \right\vert \leq 
	\frac{C}{h_{n}^{\beta }}$. This leads to%
	\begin{equation*}
		\left\vert \eta _{1}\right\vert \leq \frac{Cp_{n}}{h_{n}}\left( \alpha
		_{1}+h_{n}\alpha _{2}\right) .
	\end{equation*}%
	Then, by Tchebychev's inequality, we have 
	\begin{eqnarray*}
		\frac{\tau _{n}}{n}E\left[ \eta _{1}^{2}1_{\left\{ \left\vert \eta
			_{_{1}}\right\vert >\varepsilon \sqrt{n}\sigma \left( x\right) \right\} }%
		\right]  &\leq &\frac{C\tau _{n}p_{n}^{2}}{h_{n}^{2}n}\left( \alpha
		_{1}+h_{n}\alpha _{2}\right) ^{2}P\left( \left\vert \eta _{_{1}}\right\vert
		>\varepsilon \sqrt{n}\sigma \left( x\right) \right)  \\
		&\leq &\frac{Cp_{n}^{2}}{h_{n}^{2}n}\left( \alpha _{1}+h_{n}\alpha
		_{2}\right) ^{2}\frac{\frac{\tau _{n}}{n}\text{Var}\left( \eta _{_{1}}\right) }{%
			\varepsilon ^{2}\sigma ^{2}\left( x\right) } \\
		&=&\frac{Cp_{n}^{2}}{h_{n}n}\left( \alpha _{1}+h_{n}\alpha _{2}\right) ^{2}%
		\frac{\frac{\tau _{n}}{nh_{n}}\text{Var}\left( \eta _{_{1}}\right) }{\varepsilon
			^{2}\sigma ^{2}\left( x\right) }.
	\end{eqnarray*}
	
	The result follows from (\ref{conc}). This finishes the proof of both (%
	\ref{Lindeberg2}) and $\left( \ref{normality}\right) $.

\end{proof}

%\noindent\textbf{References}
\bibliographystyle{elsarticle-harv}
\bibliography{bib3}


%\begin{thebibliography}{00}



%\end{thebibliography}
\end{document}